\documentclass[11pt]{amsart}
\usepackage[T1]{fontenc}
\usepackage{amssymb,amsmath, amsthm, amsfonts, tikz-cd}

\usepackage{graphicx}
\usepackage{listings}
\usepackage[margin=.75in]{geometry}
\usepackage{lstautogobble}
\usepackage{enumerate}
\usepackage[shortlabels]{enumitem}
\usepackage{thmtools}
\usepackage{thm-restate}
\usepackage{amsthm}
\usepackage{verbatim}

\usepackage{mathtools}
\usepackage{physics}
\usepackage{color}
\usepackage{hyperref}
\usepackage[capitalise]{cleveref}
\usepackage[final]{showlabels}
\usepackage[bottom]{footmisc}
\crefformat{equation}{(#2#1#3)}

\usepackage{float}
\restylefloat{table}

\usepackage{mathrsfs}
\setlist{  
  listparindent=\parindent,
  parsep=0pt,
}

\theoremstyle{plain}
\newtheorem{thm}{Theorem}[section]
\newtheorem{prop}[thm]{Proposition}
\newtheorem{lemma}[thm]{Lemma}
\newtheorem{cor}[thm]{Corollary}

\theoremstyle{definition}
\newtheorem{mydef}[thm]{Definition}

\newtheorem{remark}[thm]{Remark}

\numberwithin{equation}{section} %Equation numbering

\DeclarePairedDelimiter\ipp{\langle}{\rangle}

\DeclarePairedDelimiter{\paren}{\lparen}{\rparen}

\DeclarePairedDelimiter{\jp}{\langle}{\rangle}

\DeclareMathOperator{\supp}{supp}

\newcommand{\M}{{\mathcal{M}}}

\newcommand{\p}{{\partial}}

\renewcommand{\d}{\delta}

\newcommand{\R}{{\mathbb{R}}}
\newcommand{\C}{{\mathbb{C}}}
\newcommand{\N}{{\mathbb{N}}}

\newcommand{\Z}{{\mathbb{Z}}}

\renewcommand{\P}{{\mathcal{P}}}

\newcommand{\g}{{\mathfrak{g}}}

\newcommand{\G}{{\mathfrak{G}}}
\newcommand{\f}{\mathfrak{f}}
\newcommand{\Fr}{{\mathfrak{F}}}

\newcommand{\J}{{\mathbf{J}}}

\newcommand{\I}{{\mathcal{I}}}
\newcommand{\Sc}{{\mathcal{S}}}

\renewcommand{\M}{{\mathcal{M}}}

\newcommand{\wt}{\widetilde}
\newcommand{\tl}{\tilde}

\newcommand{\D}{\Delta}

\newcommand{\ph}{\phantom{=}}
\newcommand{\nn}{\nonumber}

\newcommand{\ol}{\overline}
\newcommand{\ul}{\underline}

\newcommand{\ux}{\underline{x}}

\newcommand{\ua}{\underline{\alpha}}
\newcommand{\ue}{\underline{\eta}}
\newcommand{\ur}{\underline{r}}

\newcommand{\ep}{\epsilon}
\newcommand{\vep}{\varepsilon}
\newcommand{\al}{\alpha}
\newcommand{\et}{\eta}

\newcommand{\om}{\omega}

\newcommand{\wh}{\widehat}
\newcommand{\ueta}{\underline{\eta}}

\renewcommand{\ueta}{\underline{\eta}}

\let\oldtocsection=\tocsection
 
\let\oldtocsubsection=\tocsubsection
 
\let\oldtocsubsubsection=\tocsubsubsection
 
\renewcommand{\tocsection}[2]{\hspace{0em}\oldtocsection{#1}{#2}}
\renewcommand{\tocsubsection}[2]{\hspace{1em}\oldtocsubsection{#1}{#2}}
\renewcommand{\tocsubsubsection}[2]{\hspace{2em}\oldtocsubsubsection{#1}{#2}}

\begin{document}

\title[Coulomb Mean-Field Convergence without Regularity]{Mean-Field Convergence of Systems of Particles with Coulomb Interactions in Higher Dimensions without Regularity}

\author[M. Rosenzweig]{Matthew Rosenzweig}
\address{  
Department of Mathematics\\ 
Massachusetts Institute of Technology\\
Headquarters Office\\
Simons Building (Building 2), Room 106\\
77 Massachusetts Ave\\
Cambridge, MA 02139-4307}
\email{mrosenzw@mit.edu}

\begin{abstract}
We consider first-order conservative systems of particles with binary Coulomb interactions in the mean-field scaling regime in dimensions $d\geq 3$. We show that if at some time, the associated sequence of empirical measures converges in a suitable sense to a probability measure with density $\om^0\in L^\infty(\R^d)$ as the number of particles $N\rightarrow\infty$, then the sequence converges for short times in the weak-* topology for measures to the unique solution of the mean-field PDE with initial datum $\omega^0$. This result extends our previous work \cite{Rosenzweig2020_PVMF} for point vortices (i.e. $d=2$). In contrast to the previous work \cite{Serfaty2020}, our theorem only requires the limiting measure belong to a scaling-critical function space for the well-posedness of the mean-field PDE, in particular requiring no regularity. Our proof is based on a combination of the modulated-energy method of Serfaty \cite{Serfaty2017} and a novel mollification argument first introduced by the author in \cite{Rosenzweig2020_PVMF}.
\end{abstract}

\maketitle

\section{Introduction}\label{sec:intro}
\subsection{Coulomb systems}
We consider first-order systems of particles with binary interactions governed by the Coulomb potential:
\begin{equation}
\label{eq:Cou_sys}
\begin{cases}
\dot{x}_i(t) &= \sum_{{1\leq j\leq N}\atop{j\neq i}} a_j (\J\nabla \g)(x_i(t)-x_j(t)) \\
x_i(0) &= x_i^0
\end{cases}
\qquad i\in\{1,\ldots,N\},
\end{equation}
where $N$ is the number of particles, $a_1,\ldots,a_N\in\R\setminus\{0\}$ are the intensities or masses of the particles, $x_1^0,\ldots,x_N^0\in\R^d$ are the pair-wise distinct initial positions, $\J$ is a $d\times d$ anti-symmetric real matrix,\footnote{Note that there is no anti-symmetric real $1\times 1$ matrix, so we consider $d\geq 2$.} and $\g$ is the $d$-dimensional Coulomb potential given by
\begin{equation}
\label{eq:g_def}
\g(x) \coloneqq
\begin{cases}
-\frac{1}{2\pi}\ln|x|, & {d=2} \\
c_d|x|^{2-d}, & {d\neq 2}
\end{cases}.
\end{equation}
The $d=2$ case is of particular interest, as it is the well-known \emph{point vortex model} originating in the work of Helmholtz \cite{Helmholtz1858} and Kirchoff \cite{Kirchoff1876}, which is an idealization of two-dimensional (2D) incompressible, inviscid fluid flow \cite[Chapter 4]{MP2012book}. The $d\geq 3$ cases are the higher-dimensional generalizations. We can conveniently view the system \eqref{eq:Cou_sys} through the lens of Hamiltonian mechanics. Namely, the system is the equation of motion for the Hamiltonian
\begin{equation}
\label{eq:Cou_sys_ham}
H_{N,d}(\ux_N) \coloneqq \frac{1}{2}\sum_{1\leq i\neq j\leq N}a_ia_j\g(x_i-x_j),
\end{equation}
with respect to a Poisson bracket on $(\R^d)^N$ constructed from $\J$ and the $a_i$ (see \cite[Section 3.1]{Rosenzweig2020_PV}). In particular, the Hamiltonian (a.k.a. energy) $H_{N,d}$ is conserved. In the repulsive case, where the intensities $a_i$ are identically signed and which is the subject of this article, one can use conservation of $H_{N,d}$ to show the particles remain well-separated. This then implies the existence of a unique, global smooth solution to \eqref{eq:Cou_sys} (see \cite[Section 3.3]{Rosenzweig2020_PV} for details).

The Coulomb system \eqref{eq:Cou_sys} is intimately connected to the following scalar equation:
\begin{equation}
\label{eq:Cou_pde}
\begin{cases}
&\p_t\omega + \nabla\cdot(u\omega) = 0 \\
&u = \J\nabla \g\ast\omega \\
&\omega|_{t=0} =\omega^0
\end{cases},
\qquad (t,x) \in [0,\infty)\times\R^d.
\end{equation}
Note that the nonlinear PDE \eqref{eq:Cou_pde} bears strong resemblance to the vorticity formulation of the \emph{2D incompressible Euler equation}. Indeed, if $d=2$, and $\J$ is the clockwise rotation by $\pi/2$, we recover 2D Euler.\footnote{We emphasize that the connection to Euler breaks down at $d=3$, since equation \eqref{eq:Cou_pde} is missing the vortex-stretching term responsible for many of the challenges in studying the dynamics of 3D Euler.} We may informally view the solution $\omega$ to \eqref{eq:Cou_pde} as a continuum superposition of dynamical point masses; and in fact, the \emph{empirical measure}
\begin{equation}
\label{eq:EM}
\omega_N(t,x) \coloneqq \sum_{i=1}^N a_i\d_{x_i(t)}
\end{equation}
solves the following weak formulation of equation \eqref{eq:Cou_pde}:
\begin{equation}
\label{eq:WVF}
\begin{split}
\int_{\R^d}\omega(t,x)f(t,x)dx &= \int_{\R^d}\omega^0(x)f(0,x)dx +\int_0^t\int_{\R^d}\omega(s,x)(\p_t f)(s,x)dxds\\
&\ph + \frac{1}{2}\int_0^t \int_{(\R^d)^2\setminus\D_2}\omega(t,x)\omega(t,y) (\J\nabla\g)(x-y)\cdot [(\nabla f)(s,x) - (\nabla f)(s,y)]dxdyds,
\end{split}
\end{equation}
where $f$ is any space-time test function and $\D_2 \coloneqq \{(x,y)\in (\R^d)^2 : x=y\}$.

Similarly to \eqref{eq:Cou_sys_ham}, the PDE \eqref{eq:Cou_pde} may be formally recast as the equation of motion for the following Hamiltonian and (possibly degenerate) Poisson bracket:
\begin{equation}
\label{eq:Cou_pde_ham}
H_{d}(\omega)\coloneqq \frac{1}{2}\ipp*{\omega,\g\ast\omega}_{L^2} \qquad \pb{F}{G}(\om)\coloneqq \int_{\R^d} \J\nabla\frac{\d F}{\d\om}(x) \cdot \nabla\frac{\d G}{\d\om}(x)d\om(x),
\end{equation}
where $F,G$ are smooth, real-valued observables and $\d F/\d\om, \d G/\d\om$ are their variational derivatives. It is straightforward to check that an observable of the form $F(\om) = \int_{\R^d}f(\om(x))dx$, for arbitrary $f$, Poisson commutes with any observable $G$ (i.e. $F$ is a Casimir), which formally implies that the Hamiltonian (a.k.a. energy) $H_d$ and all the $L^p$ norms, for $1\leq p\leq\infty$, of $\om$ are conserved. By adapting the arguments as for the 2D case, one may show that classical solutions of \eqref{eq:Cou_pde} exist and are global (cf. \cite{Wolibner1933}) and weak solutions in $L^\infty$ exist and are unique and global (cf. \cite{Yudovich1963}). Note that the solution class to \eqref{eq:Cou_pde} is invariant under the scaling
\begin{equation}
\label{eq:scale}
\om(t,x) \mapsto \om(t,\lambda x),
\end{equation}
under which the $L^\infty(\R^d)$ norm is also invariant. Therefore, $L^\infty(\R^d)$ is a \emph{scaling-critical} function space for the well-posedness of \eqref{eq:Cou_pde}. For numerous kinds of PDE, critical function spaces are the threshold for a change in the equation's dynamics (e.g. global existence vs. finite-time blowup).

In recent years, there has been a strong interest in obtaining \emph{effective descriptions} of the dynamics of systems, such as \eqref{eq:Cou_sys}, when the number $N$ of particles is very large. A sense of the scope of activity on this subject is conveyed by the very nice surveys of Jabin \cite{Jab2014} and Golse \cite{Golse2016_survey}. The motivation comes from two directions.
\begin{enumerate}
\item
We can imagine that \eqref{eq:Cou_sys} is intended as a simple dynamical model for the behavior of particles or agents whose dynamics we can in principle completely determine by solving the ODEs. However, when $N\gg 1$, it may be computationally too expensive to solve such a coupled system. Thus, an effective description of the system provided by a continuum model (i.e. a PDE) becomes desirable.
\item
We can imagine that \eqref{eq:Cou_sys} is a numerical approximation scheme (part of a so-called particle method \cite{GHL1990}) for the PDE \eqref{eq:Cou_pde}, where the solution $\omega$ is approximated by the empirical measure $\omega_N$. As with any such method, we seek an estimate of the error introduced by the approximation. 
\end{enumerate}
The asymptotic regime typically considered to study this problem is the so-called \emph{mean-field regime}, where the intensities $a_i=1/N$ for every $1\leq i\leq N$, so that the velocity field experienced by the $i$-th particle is proportional to the average of the fields generated by the remaining particles. Based on the earlier observation that the empirical measure \eqref{eq:EM} is a weak solution to \eqref{eq:Cou_pde}, we expect that $\omega_N$ weak-* converges to a solution $\omega$ of the PDE \eqref{eq:Cou_pde}, as $N\rightarrow \infty$, point-wise in time:
\begin{equation}
\label{eq:MF_conv}
\forall t\geq 0, \qquad \omega_N \xrightharpoonup[N\rightarrow\infty]{*} \om_N(t).
\end{equation}

\subsection{Prior results}
To the best of our knowledge, the first rigorous result on mean-field convergence of the system \eqref{eq:Cou_sys} is by Schochet \cite{Schochet1996} for $d=2$. Using a refinement \cite{Schochet1995} of Delort's existence theorem in $H^{-1}$ \cite{Delort1991}, Schochet showed that if the initial data $(\ux_N^0)_{N\in\N}$ have uniform-in-$N$ compact support, mass, and energy and the initial empirical measures $\om_N^0 \xrightharpoonup[N\rightarrow\infty]{*}\om^0\in\M(\R^2)$, the space of signed measures, then there is a subsequence $\om_{N_k} \xrightharpoonup[N\rightarrow\infty]{*}\om$ point-wise in time, where $\om$ is \emph{some} solution in $H^{-1}$ to 2D Euler with initial datum $\omega^0$. Due to a lack of uniqueness for 2D Euler in $H^{-1}$, one does not know from Schochet's result that if $\omega^0$ is such that there is a unique solution $\omega$ with initial datum $\omega^0$ in a more a regular class (e.g. $L_t^\infty(L_x^1\cap L_x^\infty)$), then $\omega_N$ converges to this unique solution $\omega$. 

In \cite{Serfaty2020}, Serfaty gave the first complete proof in all dimensions of the validity of the mean-field limit,\footnote{This work in fact also covers a generalization of the system \eqref{eq:Cou_sys} in which the Coulomb potential is replaced by a locally integrable Riesz potential.} which has the added benefit of being quantitative. Serfaty's proof relies on the \emph{modulated-energy method}, which she previously introduced in work \cite{Serfaty2017} related to Ginzburg-Landau vortices. We defer a discussion of this method to \cref{ssec:intro_RM} below. Compared to Schochet's aforementioned partial result, Serfaty's proof requires that the initial empirical measures converge to a measure $\omega^0$ with bounded density, such that the associated unique solution $\omega\in L^\infty([0,\infty); \P(\R^d)\cap  L^\infty(\R^d))$, where $\P(\R^d)$ is the space of probability measures, satisfies the bound
\begin{equation}
\label{eq:Serf_bnd}
\sup_{0\leq t\leq T} \|\nabla^2(\g\ast\omega(t))\|_{L^\infty(\R^d)} <\infty, \qquad \forall T>0.
\end{equation}
Note that the assumption \eqref{eq:Serf_bnd} implies that the velocity field $u$ in \eqref{eq:Cou_pde} is spatially Lipschitz. Since the operator $\nabla^2\g\ast $ is (up to a constant) the matrix of Riesz transforms on $\R^d$, this operator is unbounded on $L^\infty(\R^d)$. Thus, the assumption \eqref{eq:Serf_bnd} cannot be ensured by the conservation of the $L^\infty(\R^d)$ norm and requires a stronger condition on the initial datum than $\omega^0\in L^\infty(\R^d)$ in order to ensure by known well-posedness theory (e.g. see \cite[Theorem 7.26]{BCD2011}).

Recently, the author \cite{Rosenzweig2020_PVMF} managed to drop the assumption \eqref{eq:Serf_bnd} in dimension $d=2$, thus only requiring that $\om_0\in L^\infty(\R^2)$. The proof again relies on the modulated-energy method, but to overcome the Lipschitz assumption implicit in \eqref{eq:Serf_bnd}, the author introduced a novel \emph{mollification} argument, inspired by earlier work on mean-field limits of quantum many-body systems \cite{Rosenzweig2019_LL}. We defer a discussion of this mollification argument until \cref{ssec:intro_RM} below. To the best of our knowledge, our result is the only scaling-critical proof of mean-field convergence for system \eqref{eq:Cou_sys}.

We also mention the important work of Duerinckx \cite{Duerinckx2016}, which is between Schochet's and Serfaty's results \cite{Schochet1996} and \cite{Serfaty2020}, respectively. His work treats mean-field convergence for the \emph{gradient flow} analogues of \cref{eq:Cou_sys} and \cref{eq:Cou_pde}, where the skew-symmetric matrix $\J$ is replaced by the identity, as well as the case of mixed gradient and Hamiltonian flows if $d=2$. The argument of \cite{Duerinckx2016} breaks down for the present model \cref{eq:Cou_sys}, but several of the ideas behind Duerinckx's analysis have been relevant to subsequent works, including this one.

\subsection{Overview of main results}
\label{ssec:intro_MR}
As described in the last subsection, mean-field convergence has been shown to hold in dimension $d=2$ under scaling-critical assumptions for the mean-field solution $\om$. Given this result, it is natural to ask if one can also show mean-field convergence under the same assumptions for the limiting measure $\om$ in dimensions $d\geq 3$.\footnote{The author thanks Sylvia Serfaty for suggesting this problem.} The main result of this article is an affirmative answer to this question.

\begin{thm}[Main result I]
\label{thm:main}
For $d\geq 3$, there exists an absolute constant $C_d>0$ such that the following holds. Let $\om\in L^\infty([0,\infty);\P(\R^d)\cap L^\infty(\R^d))$ be a weak solution to the equation \eqref{eq:Cou_pde} with initial datum $\om^0$. Let $N\in\N$, and let $\ux_N \in C^\infty([0,\infty); (\R^d)^N\setminus\D_N)$ be a solution to the system \eqref{eq:Cou_sys}. Define the functional
\begin{equation}
\label{eq:ME_intro}
\Fr_N^{avg}:[0,\infty)\rightarrow\R, \qquad \Fr_N^{avg}(\ux_N(t),\om(t))\coloneqq \int_{(\R^d)^N\setminus\D_2}\g(x-y)d(\om_N-\om)(t,x)d(\om_N-\om)(t,y),
\end{equation}
where $\om_N$ is the empirical measure defined in \eqref{eq:EM}. Then it holds that 
\begin{equation}
\left|\Fr_N^{avg}(\ux_N(t),\om(t))\right| \leq \G_{d,N}(t)\exp{C_{d}t\|\om^0\|_{L^\infty(\R^d)}\paren*{\ln_+(N^2H_{N,d})+\ln(N)}},
\end{equation}
where $\ln_+(\cdot) \coloneqq \max\{\ln(\cdot),0\}$ and
\begin{equation}
\G_{d,N}(t)\coloneqq |\Fr_N^{avg}(\ux_N(0),\om(0))| + C_{d}t\paren*{\|\om^0\|_{L^\infty(\R^d)} + \|\om^0\|_{L^\infty(\R^d)}^2}\frac{(\ln(N) + \ln_+(N^2H_{N,d}))}{N^{\frac{2}{d^2-2}}}.
\end{equation}
\end{thm}

Assuming that the initial modulated energy $\Fr_N^{avg}(\ux_N(0),\om(0))=O(N^{-\eta})$, for some $\eta>0$, as $N\rightarrow\infty$ (see \cref{rem:LLN} below), \cref{thm:main} implies that the empirical measure $\om_N$ converges to $\om$ as $N\rightarrow\infty$ for short times in $H^s(\R^d)$ norm, for any $s<-d/2$. From this Sobolev convergence, we deduce weak-* convergence in $\M(\R^d)$ for short times.

\begin{cor}[Main result II]
\label{cor:main}
For $d\geq 3$ and given $s<-d/2$ and $\eta>0$, there exist absolute constants $C_{d},C_{d,\et},C_{d,s,\eta}>0$ such that the following holds. If $\ux_N$ and $\om$ satisfy the assumptions of \cref{thm:main} and $|\Fr_N^{avg}(\ux_N(0),\om(0))|\lesssim_\eta N^{-\eta}$, then it holds that for all $N\gg 1$,
\begin{equation}
\begin{split}
&\|\om_N-\om\|_{L^\infty([0,T]; H^s(\R^d))} \\
&\leq C_{d,s,\eta}(1+\|\om^0\|_{L^\infty(\R^d)})^{1/2}N^{-1/3} + N^{\frac{2s+d}{3(1-s)}}\\
&\ph + C_{d,s,\eta}\paren*{N^{-\eta}+ T(\|\om^0\|_{L^\infty(\R^d)} + \|\om^0\|_{L^\infty(\R^d)}^2)\frac{(\ln(N)+\ln(H_d))}{N^{\frac{2}{d^2-2}}}}^{1/2}\\
&\hspace{15mm}\times N^{C_dT\|\om^0\|_{L^\infty(\R^d)}}e^{C_{d,\et}T\|\om^0\|_{L^\infty(\R^d)} |H_d|}.
\end{split}
\end{equation}
Consequently, if $T=T(d,\|\om_0\|_{L^\infty(\R^d)}, \eta) >0$ is sufficiently small, then $\om_N\xrightharpoonup[N\rightarrow\infty]{*}\om$ in $\M(\R^d)$ on $[0,T]$.
\end{cor}

We conclude this subsection with the following remarks on the assumptions in the statements of \cref{thm:main} and \cref{cor:main}, as well as some of the implications of these results.
\begin{remark}
\label{rem:LLN}
One can produce examples of sequences $\{\ux_N\}_{N\in\N}$ satisfying the initial modulated energy convergence assumption in \cref{cor:main} by randomizing the initial configurations (cf. \cite[Remark 1.2(c)]{Duerinckx2016}). For every $N\in\N$, take $x_{1,N}^0,\ldots,x_{N,N}^0$ to be i.i.d. $\R^d$-valued random variables with probability density function $\om^0$. Then one can show using the Kolmogorov three series theorem that a.s. $\Fr_N^{avg}(\ux_N^0,\om^0)=O(N^{-\eta})$, for some $0<\eta<1/2$.
\end{remark}

\begin{remark}
\label{rem:short_t}
The short-time restriction needed in \cref{cor:main} to obtain weak-* convergence stems from an inefficiency in the proof of \cref{lem:kprop_error}. There is no corresponding issue in dimension $d=2$ (cf. \cite[Lemma 4.2] {Rosenzweig2020_PVMF}) because the 2D Coulomb potential is logarithmic, a point we discuss more in the next subsection. We believe that this restriction is merely technical.
\end{remark}

\begin{remark}
\label{rem:ID}
If $\om \in \P(\R^d)\cap L^\infty(\R^d)$, then by \cref{lem:PE_bnds}, $\g\ast \om \in L^\infty(\R^d)$, implying that $\om$ has finite Coulomb energy, which is conserved (i.e. $H_d(\om)=H_d(\om^0)<\infty$), and the quantity $\Fr_N^{avg}(\ux_N,\om)$ is well-defined. In fact, an examination of the proof \cref{thm:main}, in particular the use of \cref{prop:kprop}, reveals that it suffices to have a weak solution in $L^\infty([0,T]; \P(\R^d)\cap L^p(\R^d))$, for some $p>d$, provided that we have an \emph{a priori} bound on the log-Lipschitz norm of the velocity field $u(t)$ uniformly in $[0,T]$. Such a bound is guaranteed if $\omega^0\in L^\infty(\R^d)$ (see \cref{lem:PE_bnds}).
\end{remark}

\begin{remark}
\label{rem:pc}
\cref{cor:main} implies \emph{propagation of chaos} for short times if the $N$-body initial configurations $\ux_N^0$ are chosen randomly according to the probability density function $f_N^0=(\om^0)^N$. This convergence follows from the Grunbaum lemma and arguing similarly as to \cite[Lemma 8.4]{RS2016}.
\end{remark}

\subsection{Road map of proofs}
\label{ssec:intro_RM}
Our proof of \cref{thm:main} builds on our earlier work \cite{Rosenzweig2020_PVMF} treating the $d=2$ case, which in turn was inspired by the modulated-energy method as developed by Duerinckx and Serfaty in the aforementioned works \cite{Duerinckx2016} and \cite{Serfaty2020}, respectively.\footnote{We also mention the interesting reecent work of Bresch, Jabin, and Wang \cite{BJW2019}, which combines the modulated-energy method of Serfaty with the relative entropy method of \cite{JW2018} to obtain a new \emph{modulated-free-energy method}.} This method exploits a weak-strong stability principle for equations of the form \eqref{eq:Cou_pde}. Its advantages include that it is quantitative and that it avoids a need for control of the microscopic dynamics in terms of particle trajectories. Until our prior work \cite{Rosenzweig2020_PVMF}, the main drawback of the modulated-energy method had been that it requires some regularity for the limiting solution $\om$ or an \emph{a priori} assumption on the velocity field $u$.

As in \cite{Duerinckx2016, Serfaty2020, Rosenzweig2020_PVMF}, the idea is to take as modulated energy the quantity $\Fr_N^{avg}(\ux_N(t),\om(t))$ from the statement of \cref{thm:main}. This quantity may be considered a renormalized $\dot{H}^{-1}(\R^d)$ semi-norm, so as to remove the infinite self-interaction between particles. One now proceeds by the energy method, obtaining the identity
\begin{equation}
\label{eq:intro_ME_td}
\frac{d}{dt}\Fr_N^{avg}(\ux_N(t),\om(t))=\int_{(\R^d)^2\setminus\D_2} (\nabla\g)(x-y)\cdot(u(t,x)-u(t,y))d(\om_N-\om)(t,x)d(\om_N-\om)(t,y) \ \text{a.e.}
\end{equation}
The challenge is to control the right-hand side point-wise in time. If $u$ is spatially Lipschitz locally uniformly in time, then one can obtain the following bound with a bit of sweat (see \cite[Proposition 2.3]{Serfaty2020}), with which one can conclude by an application of the Gronwall-Bellman inequality:
\begin{equation}
\label{eq:Lip_ME_bnd}
\begin{split}
\left|\Fr_N^{avg}(\ux_N(t),\om(t))\right| &\leq \left|\Fr_N^{avg}(\ux_N(0),\om(0))\right| + C(\|\om^0\|_{L^\infty},\|u\|_{L_t^\infty W_x^{1,\infty}})N^{-1/2} \\
&+ C(\|\om^0\|_{L^\infty},\|u\|_{L_t^\infty W_x^{1,\infty}})\int_0^t \left|\Fr_N^{avg}(\ux_N(s),\om(s))\right|ds \qquad t\in [0,T], 
\end{split}
\end{equation}
where $C_d(\cdot,\cdot)$ denotes a constant depending on the dimension $d$ and its two arguments. Removing the Lipschitz assumption implicit in \eqref{eq:Lip_ME_bnd} requires several new ideas, which are deployed in \cref{prop:kprop} (cf. \cite[Proposition 1.7]{Rosenzweig2020_PVMF}, \cite[Proposition 2.3]{Serfaty2020}). We omit the statement here, as it is long and complicated.

In \cite{Rosenzweig2020_PVMF}, we managed to remove the Lipschitz assumption through a novel mollification argument, where we replace the velocity $u$ with a mollified velocity $u_{\ep_2}$, where $\ep_2>0$ is a small parameter to be optimized. Using the log-Lipschitz regularity of $u$ (see \cref{lem:conv_bnds}), we obtain a quantitative error for the replacement. We then proceed in the modulated-energy method working with $u_{\ep_2}$, avoiding any estimates that require more regularity than log-Lipschitz. But there is one point where we have to put $\nabla u_{\ep_2}$ in $L^\infty$, leading to a large factor of $\ln\ep_2^{-1}$. Ultimately, we can absorb this factor by allowing $\ep_2$ to be time-dependent, in fact, depending on $|\Fr_N^{avg}(\ux_N(t),\om(t))|$. Although the final integral inequality we obtain has various log factors, it still satisfies the Osgood condition, allowing us to close the estimate.

For the proof of \cref{prop:kprop}, we again want to use this idea, but estimating the error from mollifying (see \cref{lem:kprop_error}) is not so straightforward and, in fact, is the source of the restriction to short times. The reason boils down to the failure of an elementary identity for the Coulomb potential in dimensions $d\geq 3$. Namely, for $|x|<e^{-1}$,
\begin{equation}
\label{eq:Cou_pot_iss}
|x\ln(|x|^{-1})\cdot(\nabla\g)(x)| \lesssim_d
\begin{cases}
\g(x), & {d=2}\\
\ln(|x|^{-1})\g(x) , & {d\geq 3}
\end{cases}.
\end{equation}
This turns out to have serious implications when we try to use Taylor's theorem to cancel one degree of singularity in the kernel of the velocity field. Ultimately, we pick up a term of the form
\begin{equation}
\int_0^t C_d\ln(N)|\Fr_N^{avg}(\ux_N(s),\om(s))|ds
\end{equation}
in our final integral inequality for the modulated energy, which can be handled provided $\Fr_N^{avg}(\ux_N(0),\om(0))=O(N^{-\eta})$, for some $\eta>0$, and that $t$ is sufficiently small, depending on $\eta$ and other fixed parameters.

\subsection{Organization of article}
\label{ssec:intro_org}
Having presented the main results of this article and discussed their proofs, we briefly comment on the organization of the body of the article.

\cref{sec:pre} consists of notation and preliminary facts from harmonic analysis, concerning singular integral estimates used extensively in \Cref{sec:kprop,sec:MR}. \cref{sec:CE} is devoted to properties of the Coulomb potential $\g$ and the modulated energy functional $\Fr_N^{avg}(\cdot,\cdot)$ from \eqref{eq:ME_intro}. Many of the results in the section are already in the articles \cite{Duerinckx2016, Serfaty2020}. Therefore, we generally include only statements of the results and skip repeating the proofs.

\cref{sec:kprop} is the meat of this article, containing the proof of \cref{prop:kprop}. We begin this section with an overview of the main steps of the proof, following the presentation in \cite[Section 4.1]{Rosenzweig2020_PVMF}. The remaining subsections correspond to the main steps of the proposition's proof. Finally, \cref{sec:MR} is where we use \cref{prop:kprop} to prove our main results, \cref{thm:main} and \cref{cor:main}.

\subsection{Acknowledgments}
\label{ssec:intro_ack}
The author thanks Sylvia Serfaty for suggesting the problem addressed in this article and for insightful correspondence. The author gratefully acknowledges financial support from the Simons Collaboration on Wave Turbulence.

\section{Preliminaries}
\label{sec:pre}
In this section, we introduce the basic notation used throughout this article and review some facts from harmonic analysis.

\subsection{Basic Notation}
\label{ssec:pre_not}

Given nonnegative quantities $A$ and $B$, we write $A\lesssim B$ if there exists a constant $C>0$, independent of $A$ and $B$, such that $A\leq CB$. If $A \lesssim B$ and $B\lesssim A$, we write $A\sim B$. To emphasize the dependence of the constant $C$ on some parameter $p$, we sometimes write $A\lesssim_p B$ or $A\sim_p B$.

We denote the natural numbers excluding zero by $\N$ and including zero by $\N_0$. Similarly, we denote the nonnegative real numbers by $\R_{\geq 0}$ and the positive real numbers by $\R_+$ or $\R_{>0}$.

Given $N\in\N$ and points $x_{1,N},\ldots,x_{N,N}$ in some set $X$, we will write $\ux_N$ to denote the $N$-tuple $(x_{1,N},\ldots,x_{N,N})$. We define the generalized diagonal $\Delta_N$ of the Cartesian product $X^N$ to be the set
\begin{equation}
\Delta_N \coloneqq \{(x_1,\ldots,x_N) \in X^N : x_i=x_j \text{ for some $i\neq j$}\}.
\end{equation}
Given $x\in\R^n$ and $r>0$, we denote the ball and sphere centered at $x$ of radius $r$ by $B(x,r)$ and $\p B(x,r)$, respectively. We denote the uniform probability measure on the sphere $\p B(x,r)$ by $\sigma_{\p B(x,r)}$. Given a function $f$, we denote the support of $f$ by $\supp(f)$. We use the notation $\jp{x}\coloneqq (1+|x|^2)^{1/2}$ to denote the Japanese bracket.

If $A=(A^{ij})_{i,j=1}^N$ and $B=(B^{ij})_{i,j=1}^N$ are two $N\times N$ matrices, with entries in $\C$, we denote their Frobenius inner product by
\begin{equation}
A : B \coloneqq \sum_{i,j=1}^N A^{ij}\ol{B^{ij}}.
\end{equation}

We denote the space of complex-valued Borel measures on $\R^n$ by $\M(\R^n)$. We denote the subspace of probability measures (i.e. elements $\mu\in\M(\R^n)$ with $\mu\geq 0$ and $\mu(\R^n)=1$) by $\P(\R^n)$. When $\mu$ is in fact absolutely continuous with respect to Lebesgue measure on $\R^n$, we shall abuse notation by writing $\mu$ for both the measure and its density function. 

We denote the Banach space of complex-valued continuous, bounded functions on $\R^n$ by $C(\R^n)$ equipped with the uniform norm $\|\cdot\|_{\infty}$. More generally, we denote the Banach space of $k$-times continuously differentiable functions with bounded derivatives up to order $k$ by $C^k(\R^n)$ equipped with the natural norm, and we define $C^\infty \coloneqq \bigcap_{k=1}^\infty C^k$. We denote the subspace of smooth functions with compact support by $C_c^\infty(\R^n)$, and use the subscript $0$ to indicate functions vanishing at infinity. We denote the Schwartz space of functions by $\Sc(\R^n)$ and the space of tempered distributions by $\Sc'(\R^n)$.

For $p\in [1,\infty]$ and $\Omega\subset\R^n$, we define $L^p(\Omega)$ to be the usual Banach space equipped with the norm
\begin{equation}
\|f\|_{L^p(\Omega)} \coloneqq \paren*{\int_\Omega |f(x)|^p dx}^{1/p}
\end{equation}
with the obvious modification if $p=\infty$. When $\Omega=\R^n$, we sometimes just write $\|f\|_{L^p}$. In the case where $f: \Omega\rightarrow X$ takes values in some Banach space $(X,\|\cdot\|_{X})$ (e.g. $L^p(\Omega;L^q(\Omega))$, we shall write $\|f\|_{L^p(\Omega;X)}$. In the special case of $L^p([0,T]; L^q(\R^n))$, shall use the abbreviation $L_t^pL_x^q([0,T]\times\R^n)$, which is justified by the Fubini-Tonelli theorem.

Our conventions for the Fourier transform and inverse Fourier transform are respectively
\begin{align}
\mathcal{F}(f)(\xi) &\coloneqq \wh{f}(\xi) \coloneqq \int_{\R^n}f(x)e^{-i\xi\cdot x}dx \qquad \forall \xi\in\R^n,\\
\mathcal{F}^{-1}(f)(x) &\coloneqq f^{\vee}(x) \coloneqq \frac{1}{(2\pi)^n}\int_{\R^n}f(\xi)e^{i x\cdot\xi}d\xi \qquad \forall x\in\R^n.
\end{align}

For integer $k\in\N_0$ and $1\leq p\leq \infty$, we define the usual Sobolev spaces
\begin{equation}
\begin{split}
W^{k,p}(\R^n) \coloneqq \{\mu \in L^p(\R^n) : \nabla^k \mu \in L^p(\R^n;(\C^n)^{\otimes k}),\quad \|\mu\|_{W^{k,p}(\R^n)} &\coloneqq \sum_{k=0}^n \|\nabla^{k}\mu\|_{L^p(\R^n)}.
\end{split}
\end{equation}
For $s\in\R$, we define the inhomogeneous Sobolev space $H^s(\R^n)$ to be the space of $\mu\in\Sc'(\R^n)$ such that $\wh{\mu}$ is locally integrable and 
\begin{equation}
\label{eq:H^s_def}
\|\mu\|_{H^s(\R^n)} \coloneqq \paren*{\int_{\R^n} \jp{\xi}^{2s} |\wh{\mu}(\xi)|^2d\xi }^{1/2}<\infty,
\end{equation}
and we use the notation $\|\mu\|_{\dot{H}^s(\R^n)}$ to denote the semi-norm where $\jp{\xi}$ is replaced by $|\xi|$.\footnote{As is standard notation, a dot superscript indicates a homogeneous semi-norm for a function space in this paper.}

\subsection{Harmonic Analysis}
\label{ssec:pre_HA}
In this subsection, we recall some basic facts from harmonic analysis concerning function spaces, Littlewood-Paley theory, and Riesz potential estimates. This material is standard in the field, and the reader can consult references such as \cite{Stein1970,Stein1993,grafakos2014c, grafakos2014m}.

\begin{mydef}[Riesz potential]
\label{def:RP}
Let $n\in\N$. For $s>-n$, we define the Fourier multiplier $(-\Delta)^{s/2}$ by
\begin{equation}
((-\Delta)^{s/2}f)(x) \coloneqq (|\cdot|^{s}\wh{f}(\cdot))^\vee(x) \qquad x\in \R^n,
\end{equation}
for a Schwartz function $f\in \Sc(\R^n)$. Since, for $s\in (-n,0)$, the inverse Fourier transform of $|\xi|^s$ is the tempered distribution
\begin{equation}
\frac{2^s\Gamma(\frac{n+s}{2})}{\pi^{\frac{n}{2}}\Gamma(-\frac{s}{2})} |x|^{-s-n},
\end{equation}
it follows that
\begin{equation}
((-\Delta)^{s/2}f)(x) = \frac{2^s \Gamma(\frac{n+s}{2})}{\pi^{\frac{n}{2}}\Gamma(-\frac{s}{2})}\int_{\R^n}\frac{f(y)}{|x-y|^{s+n}}dy, \qquad x\in\R^n.
\end{equation}
For $s\in (0,n)$, we define the \emph{Riesz potential operator} of order $s$ by $\mathcal{I}_s \coloneqq (-\Delta)^{-s/2}$ on $\Sc(\R^n)$.
\end{mydef}

$\mathcal{I}_s$ extends to a well-defined operator on any $L^p$ space, the extension also denoted by $\mathcal{I}_s$ by an abuse notation, as a consequence of the \emph{Hardy-Littlewood-Sobolev (HLS) lemma}.

\begin{prop}[Hardy-Littlewood-Sobolev]
\label{prop:HLS}
Let $n\in\N$, $s \in (0,n)$, and $1<p<q<\infty$ satisfy the relation
\begin{equation}
\frac{1}{p}-\frac{1}{q} = \frac{s}{n}.
\end{equation}
Then for all $f\in\Sc(\R^n)$,
\begin{align}
\|\mathcal{I}_s(f)\|_{L^q(\R^n)} &\lesssim_{n,s,p} \|f\|_{L^p(\R^n)},\\
\|\mathcal{I}_s(f)\|_{L^{\frac{n}{n-s},\infty}(\R^n)} &\lesssim_{n,s} \|f\|_{L^1(\R^n)},
\end{align}
where $L^{r,\infty}$ denotes the weak-$L^r$ space. Consequently, $\mathcal{I}_s$ has a unique extension to $L^p$, for all $1\leq p<\infty$.
\end{prop}

Although the HLS lemma breaks down at the endpoint case $p=\infty$ (one has a $BMO$ substitute which is not useful for our purposes), the next lemma allows us to control the $L^\infty$ norm of $\mathcal{I}_s(f)$ in terms of the $L^1$ norm and $L^p$ norm, for some $p=p(s,n)$.

\begin{lemma}[$L^{\infty}$ bound for Riesz potential]
\label{lem:Linf_RP}
For any $n\in\N$, $s\in (0,n)$, and $p\in (\frac{n}{s},\infty]$,
\begin{equation}
\|\mathcal{I}_s(f)\|_{L^\infty(\R^n)} \lesssim_{s,n,p} \|f\|_{L^{1}(\R^n)}^{1-\frac{n-s}{n(1-\frac{1}{p})}} \|f\|_{L^{p}(\R^n)}^{\frac{n-s}{n(1-\frac{1}{p})}}.
\end{equation}
\end{lemma}

We next define the Besov scale of function spaces, which first requires us to recall some basic facts from Littlewood-Paley theory. Let $\phi\in C_{c}^{\infty}(\R^n)$ be a radial, nonincreasing function, such that $0\leq \phi\leq 1$ and
\begin{equation}
\phi(x)=
\begin{cases}
1, & {|x|\leq 1}\\
0, & {|x|>2}
\end{cases}.
\end{equation}
Define the dyadic partitions of unity
\begin{align}
1&=\phi(x)+\sum_{j=1}^{\infty}[\phi(2^{-j}x)-\phi(2^{-j+1}x)] \eqqcolon \psi_{\leq 0}(x)+\sum_{j=1}^{\infty}\psi_{j}(x) \qquad \forall x\in\R^n,\\
1&=\sum_{j\in\mathbb{Z}}[\phi(2^{-j}x)-\phi(2^{-j+1}x)] \eqqcolon \sum_{j\in\mathbb{Z}}{\psi}_{j}(x) \qquad \forall x\in\R^n\setminus\{0\}.
\end{align}
For any $j\in\Z$, we define the Littlewood-Paley projector $P_{j}, P_{\leq 0}$ by
\begin{align}
(P_jf)(x) &\coloneqq (\psi_{j}(D)f)(x) = \int_{\R^n}K_{j}(x-y)f(y)dy \qquad K_j \coloneqq \psi_j^{\vee},\\
(P_{\leq 0}f)(x) &\coloneqq (\psi_{\leq 0}(D)f)(x) = \int_{\R^n} K_{\leq 0}(x-y)f(y)dy \qquad K_{\leq 0} \coloneqq \psi_{\leq 0}^{\vee}.
\end{align}

\begin{mydef}[Besov space]
\label{def:Bes}
Let $s\in\R$ and $1\leq p,q\leq\infty$. We define the inhomogeneous Besov space $B_{p,q}^s(\R^n)$ to be the space of $\mu\in\Sc'(\R^n)$ such that
\begin{equation}
\|\mu\|_{B_{p,q}^s(\R^n)} \coloneqq \paren*{\|P_{\leq 0}\mu\|_{L^p(\R^n)} + \sum_{j=1}^\infty 2^{jqs}\|P_j\mu\|_{L^p(\R^n)}^q}^{1/q} < \infty.
\end{equation}
For $p,q,s$ as above, we also define the homogeneous Besov semi-norm
\begin{equation}
\|\mu\|_{\dot{B}_{p,q}^s(\R^n)} \coloneqq \paren*{\sum_{j\in\Z} 2^{jqs} \|P_j\mu\|_{L^p(\R^n)}^{q}}^{1/q}.
\end{equation}
\end{mydef}

\begin{remark}
Any two choices of Littlewood-Paley partitions of unity used to define $\|\cdot\|_{B_{p,q}^s(\R^n)}$ (resp. $\|\cdot\|_{\dot{B}_{p,q}^s(\R^n)}$) lead to equivalent norms (resp. semi-norms). 
\end{remark}

\begin{remark}
The space $B_{2,2}^s(\R^n)$ coincides with the Sobolev space $H^s(\R^n)$, as can be seen from Plancherel's theorem. For $s\in \R_+\setminus\N$, the space $B_{\infty,\infty}^s(\R^n)$ coincides with the H\"older space $C^{[s],s-[s]}(\R^n)$ of bounded functions $\mu:\R^n\rightarrow\C$ such that $\nabla^{k}\mu$ is bounded, for integers $0\leq k\leq [s]$ and
\begin{equation}
\|\nabla^{[s]}\mu\|_{\dot{C}^{s-[s]}(\R^n)} \coloneqq \sup_{{0<|x-y|\leq 1}} \frac{|(\nabla^{[s]}\mu)(x) - (\nabla^{[s]}\mu)(y)|}{|x-y|^{s-[s]}} <\infty. 
\end{equation}
For integer $s$, the space $B_{\infty,\infty}^s(\R^n)$, sometimes called the Zygmund space of order $s$, is strictly larger than $C^s(\R^n)$.
\end{remark}

We next define the space of \emph{log-Lipschitz functions}, which are of central to importance to this work.

\begin{mydef}[Log-Lipschitz space]
\label{def:LL}
We define $LL(\R^n)$ to be the space of functions $\mu\in C(\R^n)$ such that
\begin{equation}
\|\mu\|_{LL(\R^n)} \coloneqq \sup_{0<|x-y|\leq e^{-1}} \frac{|\mu(x)-\mu(y)|}{|x-y||\ln|x-y||} <\infty.
\end{equation}
\end{mydef}

The next lemma shows that $B_{\infty,\infty}^1(\R^n)$ continuously embeds in $LL(\R^n)$.
\begin{lemma}
\label{lem:LL_embed}
It holds that
\begin{equation}
\|\mu\|_{LL(\R^n)} \lesssim_n \|\nabla \mu\|_{B_{\infty,\infty}^0(\R^n)}, \qquad \forall \mu\in B_{\infty,\infty}^1(\R^n).
\end{equation}
\end{lemma} 

The next lemma contains some potential theory estimates for the solution to Poisson's equation in dimension $d\geq 3$. We recall from the introduction that $\g(x) = c_d|x|^{2-d}$, where $c_d$ is the normalization constant, is the Coulomb potential in dimensions $d\geq 3$.

\begin{lemma}
\label{lem:PE_bnds}
Suppose that $\mu\in L^1(\R^d)\cap L^p(\R^d)$, for $d\geq 3$ and some $d/2<p<\infty$. Then the convolution $\g\ast\mu$ is a well-defined continuous function. Moreover, if $d/2<p\leq d$, then
\begin{equation}
\begin{split}
\|\g\ast\mu\|_{L^\infty(\R^d)} &\lesssim_{d,p} \|\mu\|_{L^1(\R^d)}^{1-\frac{p(d-2)}{d(p-1)}} \|\mu\|_{L^p(\R^d)}^{\frac{p(d-2)}{d(p-1)}}, \\
\|\g\ast\mu\|_{\dot{B}_{\infty,\infty}^{\frac{2p-d}{p}}(\R^d)} &\lesssim_{d,p} \|\mu\|_{L^p(\R^d)};
\end{split}
\end{equation}
and if $d<p\leq\infty$, then
\begin{equation}
\begin{split}
\|\nabla(\g\ast\mu)\|_{L^\infty(\R^d)} &\lesssim_{d,p} \|\mu\|_{L^1(\R^d)}^{1-\frac{p(d-1)}{d(p-1)}} \|\mu\|_{L^p(\R^d)}^{\frac{p(d-1)}{d(p-1)}},\\
\|\nabla(\g\ast\mu)\|_{\dot{B}_{\infty,\infty}^{\frac{p-d}{p}}(\R^d)} &\lesssim_{d,p} \|\mu\|_{L^p(\R^d)}.
\end{split}
\end{equation}
\end{lemma}
\begin{proof}
It follows from \cref{lem:Linf_RP} that $\g\ast\mu$ is continuous and satisfies the $L^\infty$ bound
\begin{equation}
\|\g\ast\mu\|_{L^\infty(\R^d)} \lesssim_{d,p} \|\mu\|_{L^1(\R^d)}^{1-\frac{p(d-2)}{d(p-1)}} \|\mu\|_{L^p(\R^d)}^{\frac{p(d-2)}{d(p-1)}}.
\end{equation}
Since the distributional Fourier transform $\wh{\g}$ coincides with the function $|\xi|^{-2}$ outside the origin, it follows from Plancherel's theorem that for any $j\in\Z$,
\begin{equation}
\wh{P_j(\g\ast\mu)}(\xi) = \psi_j(\xi)|\xi|^{-2}\wh{\mu}(\xi), \qquad \forall \xi\in\R^2\setminus\{0\}.
\end{equation}
So it follows from Young's inequality and Bernstein's lemma that
\begin{equation}
\|P_j(\g\ast\mu)\|_{L^\infty(\R^d)} \lesssim_d 2^{-2j} \|P_j\mu\|_{L^\infty(\R^d)} \lesssim_{p,d} 2^{\frac{dj}{p}-2j} \|P_j\mu\|_{L^p(\R^d)}.
\end{equation}
Multiplying both sides of the preceding inequality by $2^{\frac{(2p-d)j}{p}}$ and taking the supremum over $j\in\Z$, we conclude that
\begin{equation}
\|\g\ast\mu\|_{\dot{B}_{\infty,\infty}^{\frac{2p-d}{p}}(\R^d)} \lesssim_{p,d} \|\mu\|_{L^p(\R^d)}.
\end{equation}

Now if $d<p\leq\infty$, it's straightforward to check from integrating against a test function, the Fubini-Tonelli theorem, and integration by parts that
\begin{equation}
\nabla(\g\ast\mu)= (\nabla\g\ast\mu) = \I_{1}(\mu),
\end{equation}
with equality in the sense of distributions. By \cref{lem:Linf_RP}, the right-hand side is well-defined in $L^\infty(\R^d)$. By Bernstein's lemma and \cref{lem:Linf_RP},
\begin{align}
\|\I_1(\mu)\|_{L^\infty(\R^d)} &\lesssim_{d,p} \|\mu\|_{L^1(\R^d)}^{1-\frac{(d-1)p}{d(p-1)}} \|\mu\|_{L^p(\R^d)}^{\frac{(d-1)p}{d(p-1)}}, \nn\\
\|P_j\I_1(\mu)\|_{L^\infty(\R^d)} &\lesssim_d 2^{-j}\|P_j\mu\|_{L^\infty(\R^d)}  \lesssim_{d,p} 2^{\frac{dj}{p} - j} \|\mu\|_{L^p(\R^d)}.
\end{align}
Multiplying both sides of the inequality in the second line by $2^{\frac{(p-d)j}{p}}$ and taking the supremum over $j\in\Z$ completes the proof.
\end{proof}

We conclude this subsection with some quantitative estimates for the rate of convergence of mollification, which are quite useful for the proof of \cref{prop:kprop} in \cref{sec:kprop}.

\begin{lemma}
\label{lem:conv_bnds}
Let $\chi\in C_c^\infty(\R^n)$ such that $\chi\geq 0$, $\supp(\chi)\subset B(0,1)$, and $\int_{\R^n}\chi(x)dx=1$. For $0<\ep\ll 1$, define $\chi_\ep(x)\coloneqq \ep^{-n}\chi(x/\ep)$. and $\mu_\ep\coloneqq \mu\ast\chi_\ep$. Then for every $\mu\in LL(\R^n)$, we have the estimates
\begin{align}
\|\mu_\ep\|_{L^\infty(\R^n)} &\leq \|\mu\|_{L^\infty(\R^n)}, \label{eq:v_Linf} \\ 
\|\mu-\mu_\ep\|_{L^\infty(\R^n)} &\leq \|\mu\|_{LL(\R^n)}\ep|\ln\ep|, \label{eq:v_diff_Linf}\\
\|\nabla\mu_\ep\|_{L^\infty(\R^n)} &\lesssim_n \|\mu\|_{LL(\R^n)}|\ln\ep|. \label{eq:v_grad_Linf}
\end{align}
\end{lemma}
\begin{proof}
See \cite[Lemma 2.11]{Rosenzweig2020_PVMF}.
\end{proof}

\section{The Modulated Energy}
\label{sec:CE}
In this section, we introduce the modulated energy and its renormalization, which we use to measure the distance between the $N$-body empirical measure $\omega_N$ and the mean-field measure $\omega$. Our presentation follows that of \cite[Section 3]{Rosenzweig2020_PVMF}, except now treating the generalization to dimensions $d\geq 3$.

\subsection{Setup}
\label{ssec:CE_set}
Recall from the introduction that $\g(x)\coloneqq c_d|x|^{-d+2}$ is the $d$-dimensional Coulomb potential, for $d\geq 3$. Given $\eta>0$, we define the \emph{truncation to distance $\eta$} of $\g$ by
\begin{equation}
\label{eq:g_trunc}
\g_\et:\R^d\rightarrow (0,\infty), \qquad \g_\et(x) \coloneqq \begin{cases} \g(x), & |x|\geq \eta \\ \tl{\g}(\eta), & |x|<\eta \end{cases},
\end{equation}
where we have introduced the notation $\g(x) = \wt{\g}(|x|)$ to reflect that $\g$ is a radial function. Evidently, $\g_\et$ is a continuous function on $\R^d$ and decreases like $\g$ as $|x|\rightarrow\infty$. The next lemma provides us with some identities for the distributional gradient and Laplacian of $\g_\et$, of which we shall make heavy use in the sequel. We leave the proof to the reader as an exercise in integration by parts.

\begin{lemma}
\label{lem:g_id}
For any $\eta>0$, we have the distributional identities
\begin{align}
(\nabla\g_\et)(x) &= -\frac{c_d}{(d-2)} \frac{x}{|x|^d} 1_{\geq \eta}(x), \label{eq:g_eta_grad_id}\\
(\D\g_\et)(x) &= -\sigma_{\p B(0,\eta)}(x), \label{eq:g_eta_lapl_id}
\end{align}
where $\sigma_{\p B(0,\eta)}$ is the uniform probability measure on the surface $\p B(0,\eta)\subset\R^d$.
\end{lemma}

With \cref{lem:g_id}, we can define the \emph{smearing to scale $\eta$} of the Dirac mass $\delta_0$ by
\begin{equation}
\label{eq:delta_smear}
\d_0^{(\eta)} \coloneqq -\Delta \g_\et = \sigma_{\p B(0,\eta)}.
\end{equation}
A useful identity satisfied by $\d_0^{(\eta)}$ is
\begin{equation}
\label{eq:g_conv_smear}
(\g\ast \d_0^{(\eta)})(x) = \g_\et(x),
\end{equation}
which follows from the definition \eqref{eq:delta_smear}, the associativity and commutativity of convolution, and the fact that $\g$ is a fundamental solution of the operator $-\D$.

Next, given parameters $\infty>\eta,\alpha>0$, we define the function
\begin{equation}
\label{eq:f_et_al_def}
\f_{\eta,\alpha}(x) \coloneqq (\g_{\alpha}-\g_\eta)(x), \qquad \forall x\in\R^d.
\end{equation}
From identity \eqref{eq:g_conv_smear} and the bilinearity of convolution, we see that $\f_{\et,\al}$ satisfies the identity
\begin{equation}
\label{eq:f_conv_smear}
\f_{\eta,\alpha} = \g\ast (\d_0^{(\alpha)}-\d_0^{(\eta)}).
\end{equation}
Moreover, for $\alpha>\eta$, we find by direct computation that
\begin{equation}
\label{eq:f_et_al_id}
\f_{\eta,\alpha}(x) = \begin{cases}0, & {|x|\geq\alpha}\\ \\ \frac{c_d}{\alpha^{d-2}} - \frac{c_d}{|x|^{d-2}}, & {\eta\leq|x|\leq\alpha} \\ \\ \frac{c_d}{\alpha^{d-2}}-\frac{c_d}{\eta^{d-2}}, & {|x|<\eta} \end{cases}
\end{equation}
and
\begin{equation}
\label{eq:f_grad_et_al_id}
(\nabla\f_{\eta,\alpha})(x) = \frac{c_d}{(d-2)}\frac{x}{|x|^{d}}1_{\eta\leq\cdot\leq\alpha}(x),
\end{equation}
with equality in the sense of distributions. The next lemma provides useful estimates for the $L^p$ norms of $\f_{\eta,\alpha}, \nabla\f_{\eta_\alpha}$. We shall use these estimates extensively in \cref{sec:kprop}.

\begin{lemma}
\label{lem:f_et_al_bnds}
For any $1\leq p < d/(d-2)$ and $\infty>\alpha>\eta>0$, we have that
\begin{equation}
\label{eq:f_et_al_bnd}
\|\f_{\eta,\alpha}\|_{L^p(\R^d)} \lesssim_{p,d} \al^{2-\frac{d(p-1)}{p}};
\end{equation}
and for any $1\leq p < d/(d-1)$,
\begin{equation}
\label{eq:f_grad_et_al_bnd}
\|\nabla\f_{\eta,\alpha}\|_{L^p(\R^d)} \lesssim_{p,d} \al^{1-\frac{d(p-1)}{p}}.
\end{equation}
\end{lemma}
\begin{proof}
We first show estimate \eqref{eq:f_et_al_bnd}. Using identity \eqref{eq:f_et_al_id} and polar coordinates, we find that
\begin{align}
\|\f_{\eta,\al}\|_{L^p(\R^d)}^p &= c_d'\int_0^\eta \left|\eta^{2-d}-\al^{2-d}\right|^p r^{d-1} dr + c_d'\int_\eta^\al \left|r^{2-d}-\al^{2-d}\right|^p r^{d-1} dr \nn\\
&\lesssim_d |\eta^{2-d}-\al^{2-d}|^p \eta^{d} + \int_\eta^\al r^{d-1-p(d-2)} dr \nn\\
&\lesssim_{p,d} \al^{d-p(d-2)}.
\end{align}
Taking $p$-th roots of both sides yields \eqref{eq:f_et_al_bnd}.

Proceeding similarly, as to before
\begin{align}
\|\nabla\f_{\et,\al}\|_{L^p(\R^d)}^p &= c_d'\int_\eta^\al r^{d-1-p(d-1)}dr \nn\\
&\lesssim_{d,p} \al^{d-p(d-1)} - \et^{d-p(d-1)}.
\end{align}
Taking $p$-th roots of both sides yields \eqref{eq:f_grad_et_al_bnd}.
\end{proof}

Next, given a probability measure $\mu$ with density in $L^p(\R^d)$, for some $p>d/2$, a vector $\ux_N \in (\R^d)^N$, and vector $\ue_N\in (\R_+)^N$, we define the quantities
\begin{align}
H_N^{\mu,\ux_N} &\coloneqq \g \ast (\sum_{i=1}^N\d_{x_i}-N\mu), \label{eq:HN_def}\\
H_{N,\ue_N}^{\mu,\ux_N} &\coloneqq \g\ast (\sum_{i=1}^N\d_{x_i}^{(\eta_i)}-N\mu) \label{eq:HN_trun_def},
\end{align}
where $\d_{x_i}^{(\eta_i)} = \d_{0}^{(\eta_i)}(\cdot-x_i)$.

\subsection{Energy Functional}
\label{ssec:CE_EF}
For a vector $\ux_N\in (\R^d)^N$ and a measure $\mu\in \P(\R^d)\cap L^p(\R^d)$, for some $d/2<p\leq\infty$, we define the functional
\begin{equation}
\label{eq:def_EN}
\Fr_N(\ux_N,\mu) \coloneqq \int_{(\R^d)^2\setminus \D_2}\g(x-y)d(\sum_{i=1}^N\d_{x_i}-N\mu)(x)d(\sum_{i=1}^N\d_{x_i}-N\mu)(y)
\end{equation}
where $\D_2\coloneqq \{(x,y)\in (\R^d)^2 : x=y\}$. Note that $\Fr_N(\ux_N,\mu) = N^2\Fr_N^{avg}(\ux_N,\mu)$, where $\Fr_N^{avg}(\ux_N,\mu)$ was defined in \eqref{eq:ME_intro}. The reader can check from that $\Fr_N(\ux_N,\mu)$ is well-defined, given our assumptions on $\mu$. This functional serves as a $\dot{H}^{-1}(\R^d)$ squared-distance that has been renormalized so as to remove the infinite-self interaction between elements of $\ux_N$. Our first lemma, the proof of which we omit, computes the Coulomb energy of the smeared point mass $\d_{0}^{(\eta)}$.

\begin{lemma}
\label{lem:g_smear_si}
For any $0<\eta<\infty$, we have that
\begin{equation}
\int_{(\R^d)^2}\g(x-y)d\d_0^{(\eta)}(x)d\d_0^{(\eta)}(y) = \tl{\g}(\eta).
\end{equation}
\end{lemma}

The next lemma re-expresses the Coulomb energy of $N\mu-\sum_{i=1}^N\d_{x_i}^{(\eta_i)}$ as the $L^2$ norm of the gradient of potential energy. Recall the definition of $H_{N,\ul{\eta}_N}^{\mu,\ux_N}$ from \eqref{eq:HN_trun_def}.

\begin{lemma}
\label{lem:fin_en}
Fix $N\in\N$. Let $\mu\in \P(\R^d)\cap L^p(\R^d)$, for some $d/2<p\leq \infty$, and let $\ux_N\in (\R^d)^N\setminus\D_N$. Then for any $\ue_N\in (\R_+)^N$, we have the identity
\begin{equation}
\label{eq:renorm_nts}
\begin{split}
&\int_{(\R^d)^2}\g(x-y)d(N\mu-\sum_{i=1}^N\d_{x_i}^{(\eta_i)})(x)d(N\mu-\sum_{i=1}^N\d_{x_i}^{(\eta_i)})(y) =\int_{\R^d} |(\nabla H_{N,\ul{\eta}_N}^{\mu,\ux_N})(x)|^2dx.
\end{split}
\end{equation}
\end{lemma}
\begin{proof}
See the beginning of the proof of \cite[Proposition 3.3]{Serfaty2020}.
\end{proof}

The next proposition is essentially proven in \cite[Section 2.1]{PS2017} and \cite[Section 5]{Serfaty2020} in the greater generality of Riesz, not just Coulomb, interactions. We include a self-contained proof specialized to our setting. Our proposition is the higher-dimensional analogue of \cite[Proposition 3.5]{Rosenzweig2020_PVMF}, in which we treated the 2D Coulomb case and whose proof we follow closely here.

\begin{prop}
\label{prop:CE}
Let $\mu \in \P(\R^d)\cap L^p(\R^d)$, for some $d/2<p\leq\infty$, and let $\ux_N\in (\R^d)^N\setminus\D_N$. Then
\begin{equation}
\label{eq:EN_renorm_lim}
\Fr_N(\ux_N,\mu) = \lim_{|\ul{\eta}_N|\rightarrow 0} \paren*{\int_{\R^d}|(\nabla H_{N,\ue_N}^{\mu,\ux_N})(x)|^2dx - \sum_{i=1}^N\tl{\g}(\eta_i)}
\end{equation}
and there exists a constant $C_{p,d}>0$, such that
\begin{equation}
\label{eq:EN_renorm_bnd}
\begin{split}
\sum_{1\leq i\neq j\leq N} \paren*{\g(x_i-x_j)-\tl{\g}(\eta_i)}_{+} &\leq \Fr_N(\ux_N,\mu) - \paren*{\int_{\R^d}|(\nabla H_{N,\ul{\eta}_N}^{\mu,\ux_N})(x)|^2dx - \sum_{i=1}^N\tl{\g}(\eta_i)} \\
&\ph + N\sum_{i=1}^N \et_i^{(2p-d)/p},
\end{split}
\end{equation}
where $(\cdot)_{+}\coloneqq \max\{\cdot,0\}$.
\end{prop}
\begin{proof}
We start by proving \eqref{eq:EN_renorm_lim}. We first show that
\begin{equation}
\begin{split}
&\int_{(\R^d)^2\setminus\D_2}\g(x-y)d(N\mu-\sum_{i=1}^N\d_{x_i})(x)d(N\mu-\sum_{i=1}^N\d_{x_i})(y) \\
&= \lim_{|\ul{\eta}_N|\rightarrow 0} \paren*{\int_{(\R^d)^2}\g(x-y)d(N\mu-\sum_{i=1}^N\d_{x_i}^{(\eta_i)})(x)d(N\mu-\sum_{i=1}^N\d_{x_i}^{(\eta_i)})(y) - \sum_{i=1}^N \tl{\g}(\eta_i)}.
\end{split}
\end{equation}
To see this, observe that the left-hand side of the preceding equality can be re-written as
\begin{equation}
\sum_{1\leq i\neq j\leq N} \g(x_i-x_j) + N^2\int_{(\R^d)^2}\g(x-y)d\mu(x)d\mu(y) -2N\sum_{i=1}^N (\g\ast\mu)(x_i).
\end{equation}
It now follows from the dominated convergence theorem that
\begin{equation}
\begin{split}
&\sum_{1\leq i\neq j\leq N}\g(x_i-x_j) -2N\sum_{i=1}^N(\g\ast\mu)(x_i) \\
&= \lim_{|\ul{\eta}_N|\rightarrow 0} \paren*{\sum_{1\leq i\neq j\leq N}\int_{(\R^d)^2}\g(x-y)d\d_{x_i}^{(\eta_i)}(x)d\d_{x_j}^{(\eta_j)}(y) - 2N\sum_{i=1}^N \int_{\R^d} (\g\ast\mu)(x)d\d_{x_i}^{(\eta_i)}(x)}.
\end{split}
\end{equation}
We now obtain the desired conclusion by applying \cref{lem:g_smear_si} to the first term in the right-hand side.

Next, we show the bound \eqref{eq:EN_renorm_bnd}. Fix $\ul{\eta}_N\in (\R_{+})^N$, and let $\ua_N\in (\R_+)^N$, such that $\alpha_i\ll \eta_i$ for every $i\in\{1,\ldots,N\}$ (ultimately, we let $\alpha_i\rightarrow 0^+$ for each $i$). The reader may check that
\begin{equation}
\nabla H_{N,\ul{\eta}_N}^{\mu,\ux_N} = \nabla H_{N,\ua_N}^{\mu,\ux_N} + \sum_{i=1}^N (\nabla\f_{\alpha_i,\eta_i})(\cdot-x_i).
\end{equation}
Using this identity together with a little algebra, we find that
\begin{equation}
\label{eq:main_id}
\begin{split}
\int_{\R^d}|(\nabla H_{N,\ul{\eta}_N}^{\mu,\ux_N})(x)|^2dx &= \int_{\R^d} |(\nabla H_{N,\ua_N}^{\mu,\ux_N})(x)|^2dx + 2\sum_{i=1}^N\int_{\R^d}(\nabla H_{N,\ua_N}^{\mu,\ux_N})(x)\cdot (\nabla\f_{\alpha_i,\eta_i})(x-x_i)dx \\
&\ph+ \int_{\R^d} |\sum_{i=1}^N (\nabla\f_{\alpha_i,\eta_i})(x-x_i)|^2dx.
\end{split}
\end{equation}
We consider the second and third terms in the right-hand side of the preceding equality. We expand the square to obtain
\begin{equation}
\int_{\R^d}\left|\sum_{i=1}^N (\nabla\f_{\alpha_i,\eta_i})(x-x_i)\right|^2dx = \sum_{i,j=1}^N \int_{\R^d}(\nabla\f_{\alpha_i,\eta_i})(x-x_i)\cdot(\nabla\f_{\alpha_j,\eta_j})(x-x_j)dx.
\end{equation}
Recalling the identities \eqref{eq:f_et_al_id} and \eqref{eq:f_grad_et_al_id} and integrating by parts, we obtain that
\begin{equation}
\label{eq:renorm_comb1}
\sum_{i,j=1}^N \int_{\R^d}(\nabla\f_{\alpha_i,\eta_i})(x-x_i)\cdot(\nabla\f_{\alpha_j,\eta_j})(x-x_j)dx = \sum_{i,j=1}^N \int_{\R^d} \f_{\alpha_i,\eta_i}(x-x_i)d(\d_{x_j}^{(\eta_j)}-\d_{x_j}^{(\alpha_j)})(x).
\end{equation}
Next, using the identity $-\D H_{N,\ua_N}^{\mu,\ux_N} = \sum_{i=1}^N\d_{x_i}^{(\alpha_i)} - N\mu$ and integrating by parts, we find that
\begin{equation}
\label{eq:renorm_comb2}
2\sum_{i=1}^N\int_{\R^d}(\nabla H_{N,\ua_N}^{\mu,\ux_N})(x)\cdot (\nabla\f_{\alpha_i,\eta_i})(x-x_i)dx  = 2\sum_{i=1}^N\int_{\R^d}\f_{\alpha_i,\eta_i}(x-x_i)d(\sum_{i=1}^N\d_{x_i}^{(\alpha_i)}-N\mu)(x).
\end{equation}
Summing identities \eqref{eq:renorm_comb1} and \eqref{eq:renorm_comb2}, it follows that
\begin{align}
&\sum_{i,j=1}^N \int_{\R^d}(\nabla\f_{\alpha_i,\eta_i})(x-x_i)\cdot(\nabla\f_{\alpha_j,\eta_j})(x-x_j)dx +\int_{\R^d}\left|\sum_{i=1}^N (\nabla\f_{\alpha_i,\eta_i})(x-x_i)\right|^2dx \nn\\
&=\sum_{1\leq i\neq j\leq N}\int_{\R^d}\f_{\alpha_i,\eta_i}(x-x_i)d(\d_{x_j}^{(\eta_j)}+\d_{x_j}^{(\alpha_j)})(x) + \sum_{i=1}^N\int_{\R^d}\f_{\alpha_i,\eta_i}(x-x_i)d(\d_{x_i}^{(\eta_i)}+\d_{x_i}^{(\alpha_i)})(x) \nn\\
&\ph-2N\sum_{i=1}^N\int_{\R^d}\f_{\alpha_i,\eta_i}(x-x_i)\mu(x)dx \nn\\
&\eqqcolon \mathrm{Term}_1 + \mathrm{Term}_2 + \mathrm{Term}_3.
\end{align}
We now proceed to estimate each of the $\mathrm{Term}_j$ individually, as itemized below.
\begin{itemize}[leftmargin=*]
\item
For $\mathrm{Term}_3$, H\"older's inequality (with $d/2<p\leq\infty$) and \cref{lem:f_et_al_bnds} imply that
\begin{align}
\label{eq:en_T3_fin}
\left|\mathrm{Term}_3\right| \lesssim N\sum_{i=1}^N \|\f_{\alpha_i,\eta_i}\|_{L^{p'}(\R^d)} \|\mu\|_{L^p(\R^d)} \lesssim_{p,d} N\sum_{i=1}^N \et_i^{(2p-d)/p}.
\end{align}
\item
For $\mathrm{Term}_2$, we observe from a change of variable and the definition \eqref{eq:f_et_al_def} of $\f_{\alpha_i,\eta_i}$ that
\begin{align}
\mathrm{Term}_2 &= \sum_{i=1}^N\int_{\R^d}\f_{\alpha_i,\eta_i}(x)d(\d_0^{(\eta_i)}+\d_0^{(\alpha_i)})(x) \nn\\
&= \sum_{i=1}^N\int_{\R^d}\paren*{\g_{\eta_i}(x)d\d_0^{(\eta_i)}(x)-\g_{\alpha_i}(x)d\d_0^{(\alpha_i)}(x) +\g_{\eta_i}(x)d\d_0^{(\alpha_i)}(x) - \g_{\alpha_i}(x)d\d_0^{(\eta_i)}(x)} \nn\\
&\eqqcolon\mathrm{Term}_{2,1} + \cdots+\mathrm{Term}_{2,4}.
\end{align}
Since for any $r>0$, $\d_0^{(r)}$ is the canonical probability measure on the $(d-1)$-sphere $\p B(0,r)$, $\g_{r}\equiv \tl{\g}(r)$ on $\ol{B}(0,r))$, and $\al_i<\et_i$, we find that
\begin{align}
\mathrm{Term}_{2,3} + \mathrm{Term}_{2,4} &=0, \\
\mathrm{Term}_{2,1} + \mathrm{Term}_{2,2} &= \sum_{i=1}^N \paren*{\tl{\g}(\eta_i)-\tl{\g}(\alpha_i)} \\
\mathrm{Term}_2 &= \sum_{i=1}^N \paren*{\tl{\g}(\eta_i)-\tl{\g}(\alpha_i)}. \label{eq:en_T2_fin}
\end{align}
\item
Finally, for $\mathrm{Term}_1$, we use that $\f_{\alpha_i,\eta_i}\leq 0$, a property which is evident from the identity \eqref{eq:f_et_al_id}, to estimate
\begin{align}
\mathrm{Term}_1 \leq \sum_{1\leq i\neq j\leq N}\int_{\R^d}\f_{\alpha_i,\eta_i}(x-x_i)d\d_{x_j}^{(\alpha_j)}(x) = \sum_{1\leq i\neq j\leq N}\int_{\R^d}\paren*{\g_{\eta_i}(x-x_i)-\g_{\alpha_i}(x-x_i)}d\d_{x_j}^{(\alpha_j)}(x),
\end{align}
where the penultimate expression follows from unpacking the definition of each $\f_{\alpha_i,\eta_i}$. Making a change of variable, we observe that for $1\leq i\neq j\leq N$,
\begin{align}
\int_{\R^d}\paren*{\g_{\eta_i}(x-x_i)-\g_{\alpha_i}(x-x_i)}d\d_{x_j}^{(\alpha_j)}(x) &=\int_{\R^d} \paren*{\tl{\g}_{\eta_i}(|x_j-x_i+y|) - \tl{\g}_{\alpha_i}(|x_j-x_i+y|)}d\d_{0}^{(\alpha_j)}(y).
\end{align}
Since $\tl{\g}_{\eta_i}, \tl{\g}_{\alpha_i}$ are nonincreasing, it follows from the triangle inequality that for any $y\in \p B(0,\alpha_j)$,
\begin{equation}
\paren*{\tl{\g}_{\eta_i}-\tl{\g}_{\alpha_i}}(|x_j-x_i+y|) \leq \begin{cases}0, & |x_j-x_i+y| \geq \eta_i\\ \tl{\g}(\eta_i) - \tl{\g}_{\alpha_i}(|x_j-x_i|+\alpha_j), & |x_j-x_i+y| < \eta_i \end{cases}.
\end{equation}
Since we also have that
\begin{equation}
\tl{\g}(\eta_i) - \tl{\g}_{\alpha_i}(|x_j-x_i+y|) \leq \begin{cases} \tl{\g}(\eta_i) - \tl{\g}(|x_j-x_i+y|)\leq 0, & {\alpha_i\leq |x_j-x_i+y| <\eta_i} \\ \tl{\g}(\eta_i)-\tl{\g}(\alpha_i)\leq 0, & {|x_j-x_i+y| < \alpha_i} \end{cases}
\end{equation}
by assumption that $\alpha_i<\eta_i$, it follows that
\begin{equation}
\label{eq:en_T1_fin}
\mathrm{Term}_1 \leq -\sum_{1\leq i\neq j\leq N} \paren*{\tl{\g}(|x_j-x_i|+\alpha_j)-\tl{\g}(\eta_i)}_{+}.
\end{equation}
\end{itemize}

Returning to the equation \eqref{eq:main_id} and combining identities \eqref{eq:en_T1_fin}, \eqref{eq:en_T2_fin}, and \eqref{eq:en_T3_fin} for $\mathrm{Term}_1$, $\mathrm{Term}_2$, and $\mathrm{Term}_3$, respectively, we find that there exists a constant $C_{p,d}>0$ such that
\begin{equation}
\label{eq:rearr}
\begin{split}
\int_{\R^d}|(\nabla H_{N,\ul{\eta}_N}^{\mu,\ux_N})(x)|^2dx &\leq \int_{\R^d} |(\nabla H_{N,\ua_N}^{\mu,\ux_N})(x)|^2dx  + C_{p,d}N\|\mu\|_{L^p(\R^d)}\sum_{i=1}^N\et_i^{(2p-d)/p}\\
&\ph +\sum_{i=1}^N \paren*{\tl{\g}(\eta_i)-\tl{\g}(\alpha_i)} -  \sum_{1\leq i\neq j\leq N} \paren*{\tl{\g}(|x_j-x_i|+\alpha_j)-\tl{\g}(\eta_i)}_{+}.
\end{split}
\end{equation}
Rearranging inequality \eqref{eq:rearr}, we obtain the inequality
\begin{equation}
\begin{split}
&\sum_{1\leq i\neq j\leq N} \paren*{\tl{\g}(|x_j-x_i|+\alpha_j)-\tl{\g}(\eta_i)}_{+} \\
&\leq \paren*{\int_{\R^d} |(\nabla H_{N,\ua_N}^{\mu,\ux_N})(x)|^2dx - \sum_{i=1}^N \tl{\g}(\alpha_i)} - \paren*{\int_{\R^d}|(\nabla H_{N,\ul{\eta}_N}^{\mu,\ux_N})(x)|^2dx - \sum_{i=1}^N \tl{\g}(\eta_i)} \\
&\ph + C_{p,d}N\|\mu\|_{L^p(\R^d)}\sum_{i=1}^N\et_i^{(2p-d)/p}.
\end{split}
\end{equation}
Sending $|\ua_N|\rightarrow 0$ and using identity \cref{eq:EN_renorm_lim} and that $\tl{\g}\in C_{loc}(\R_+)$ completes the proof.
\end{proof}

The next corollary is a generalization of \cite[Corollary 3.4]{Serfaty2020}, in particular a relaxation of the $\mu\in L^\infty(\R^d)$ assumption. The proof of \cite[Corollary 3.6]{Rosenzweig2020_PVMF}, which treated the 2D Coulomb case, carries over \emph{mutatis mutandis}, and therefore we omit the proof for the higher-dimensional case.

\begin{cor}
\label{cor:grad_H}
Let $d\geq 3$ and $N\in\N$. Let $\mu\in \P(\R^d)\cap L^p(\R^d)$, for some $d/2<p\leq \infty$, and let $\ux_N\in (\R^d)^N\setminus \D_N$. If for any $0<\ep_1 \ll 1$, we define
\begin{equation}
r_{i,\ep_1} \coloneqq \min\{\frac{1}{4}\min_{{1\leq j\leq N}\atop {j\neq i}} |x_i-x_j|, \ep_1\} \quad \text{and} \quad  \ur_{N,\ep_1}\coloneqq (r_{1,\ep_1},\ldots,r_{N,\ep_1}),
\end{equation}
then there exists a constant $C_{p,d}>0$ such that
\begin{equation}
\label{eq:g_r_sum_bnd}
\sum_{i=1}^N \tl{\g}(r_{i,\ep_1}) \leq \Fr_N(\ux_N,\mu) + 2N\tl{\g}(\ep_1) + C_{p,d}N^2\|\mu\|_{L^p(\R^d)}\ep_1^{(2p-d)/p}
\end{equation}
and
\begin{equation}
\label{eq:grad_H_r_bnd}
\int_{\R^d}|(\nabla H_{N,\ul{r}_{N,\ep_1}}^{\mu,\ux_N})(x)|^2dx \leq \Fr_N(\ux_N,\mu) + N\tl{\g}(\ep_1) + C_{p,d}N^2\|\mu\|_{L^p(\R^d)}\ep_1^{(2p-d)/p}.
\end{equation}
\end{cor}

\subsection{Counting Lemma}
\label{ssec:CE_CL}
In this subsection, we show how to control in terms of the modulated energy the portion of the $N$-body Hamiltonian $H_{N,d}$ coming from pairs of ``close'' particles. A similar result was implicit in the proof of \cite[Proposition 2.3]{Serfaty2020}; however, we need our refined version in the proofs of \Cref{lem:kprop_error,lem:kprop_recomb}, which are part of the proof of \cref{prop:kprop}, below.

\begin{lemma}
\label{lem:count}
Let $d\geq 3$, $N\in\N$, and $d/2< p\leq \infty$. Then there exists a constant $C_{p,d}>0$, such that for any $\ux_N\in (\R^d)^N\setminus\D_N$ and $\mu\in\P(\R^d)\cap L^p(\R^d)$, it holds that
\begin{equation}
\begin{split}
\sum_{{1\leq i\neq j\leq N}\atop{|x_i-x_j|\leq\ep_3}} \g(x_i-x_j) &\lesssim \Fr_N(\ux_N,\mu) + N\tl{\g}(2\ep_3) + C_{p,d} N^2 \|\mu\|_{L^p(\R^d)}\ep_3^{(2p-d)/p}
\end{split}
\end{equation}
for any $0<\ep_3\ll 1$.
\end{lemma}
\begin{proof}
Fix $\ep_3>0$. For every $i\in\{1,\ldots,N\}$, we chose $\eta_i = 2\ep_3$ in \cref{prop:CE} to obtain
\begin{align}
&\sum_{{1\leq i\neq j\leq N}\atop{|x_i-x_j|\leq\ep_3}}(\underbrace{\g(x_i-x_j)-\tl{\g}(2\ep_3)}_{\geq \g(x_i-x_j)/2})_{+} \nn\\
&\leq\Fr_N(\ux_N,\mu) - \int_{\R^d}|(\nabla H_{N,\ul{2\ep_3}_N}^{\mu,\ux_N})(x)|^2dx + \sum_{i=1}^N\tl{\g}(2\ep_3)+ C_{p,d}N \|\mu\|_{L^p(\R^d)}\sum_{i=1}^N(2\ep_3)^{(2p-d)/p} \nn\\
&\leq \Fr_N(\ux_N,\mu) + N\tl{\g}(2\ep_3) + C_{p,d}' N^2 \|\mu\|_{L^p(\R^d)} {\ep_3}^{(2p-d)/p},
\end{align}
where $C_{p,d}'\geq C_{p,d}$.
\end{proof}

\subsection{Coerciveness of the Energy}
\label{ssec:CE_coer}
In this final subsection, we show that the functional $\Fr_N^{avg}(\ux_N,\mu)$ controls the Sobolev norm $H^{s}(\R^d)$, for any $s<-d/2$, up to a small additive error, and convergence in the weak-* topology on $\M(\R^d)$. The following proposition is an endpoint improvement of \cite[Proposition 3.5]{Serfaty2020} (cf. \cite[Proposition 3.10]{Rosenzweig2020_PVMF}).

\begin{prop}
\label{prop:CE_coer}
Let $d\geq 3$, $N\in\N$, and $\ux_N\in (\R^d)^N\setminus\D_N$. Then for any $\mu\in\P(\R^d)\cap L^p(\R^d)$, for some $d/2<p\leq\infty$, we have the estimate
\begin{equation}
\label{eq:prop_CE_coer_Bes}
\begin{split}
\left|\int_{\R^d}\varphi(x)d(\sum_{j=1}^N\d_{x_j}-N\mu)(x)\right| &\lesssim_{d,p} N\paren*{\frac{\ep_1\|\varphi\|_{B_{2,\infty}^0(\R^d)}}{\ep_2^{(2+d)/2}} + \sum_{k\geq|\log_2\ep_2|} 2^{\frac{kd}{2}}\|P_k\varphi\|_{L^2(\R^d)}} \\
&\ph + \|\nabla\varphi\|_{L^2(\R^d)}\paren*{\Fr_N(\ux_N,\mu)+N\ep_1^{-1}+\|\mu\|_{L^p(\R^d)}N^2\ep_1^{\frac{2p-d}{p}}}^{1/2}
\end{split}
\end{equation}
for all $\varphi\in B_{2,1}^{d/2}(\R^d)$ and parameters $0<\ep_1<\ep_2\ll 1$. Consequently, for any $s<-d/2$,
\begin{equation}
\label{eq:prop_CE_coer_Sob}
\|\mu-\frac{1}{N}\sum_{i=1}^N\d_{x_i}\|_{H^s(\R^d)} \lesssim_{s,d,p} |\Fr_N^{avg}(\ux_N,\mu)|^{1/2} + (1+\|\mu\|_{L^p(\R^d)})^{1/2} N^{-\frac{2p-d}{2(3p-d)}} + N^{\frac{(2p-d)(2s+d)}{2(3p-d)(1-s)}}
\end{equation}
and if $\Fr_N^{avg}(\ux_N,\mu)\rightarrow 0$, as $N\rightarrow\infty$, then
\begin{equation}
\frac{1}{N}\sum_{i=1}^N\d_{x_i} \xrightharpoonup[N\rightarrow\infty]{*} \mu \ \text{in} \ \M(\R^d).
\end{equation}
\end{prop}
\begin{proof}
Let $\ueta_N\in (\R_+)^N$ be a parameter vector, the value of which we choose momentarily. For $\varphi\in B_{2,1}^{d/2}(\R^d)$, write
\begin{equation}
\int_{\R^d}\varphi(x)d(\sum_{j=1}^N\d_{x_j}-N\mu)(x) = \sum_{j=1}^N\int_{\R^d}\varphi(x)d(\d_{x_j}-\d_{x_j}^{(\eta_j)})(x)+\int_{\R^d}\varphi(x)d(\sum_{j=1}^N\d_{x_j}^{(\eta_j)}-N\mu)(x).
\end{equation}
For $1\leq j\leq N$, we introduce a parameter $0<\al_j\ll 1$ and make the Littlewood-Paley decomposition
\begin{equation}
P_{\leq |\log_2\alpha_j|}\varphi + P_{>|\log_2\alpha_j|}\varphi.
\end{equation}
For the low-frequency term, we have by the mean-value theorem,
\begin{equation}
\left|\int_{\R^d}(P_{\leq |\log_2\al_j|}\varphi)(x)d(\d_{x_j}-\d_{x_j}^{(\eta_j)})(x)\right| \lesssim_d \eta_j \|\nabla P_{\leq |\log_2\al_j|}\varphi\|_{L^\infty(\R^d)}.
\end{equation}
By elementary Littlewood-Paley theory,
\begin{equation}
\|\nabla P_{\leq |\log_2\al_j|}\varphi\|_{L^\infty(\R^d)} \lesssim_d \|P_{\leq 0}\varphi\|_{L^2(\R^d)} + \sum_{k=1}^{|\log_2\al_j|} 2^{k+\frac{dk}{2}} \|P_k\varphi\|_{L^2(\R^d)} \lesssim_d \frac{\|\varphi\|_{B_{2,\infty}^0(\R^d)}}{\al_j^{(2+d)/2}}, 
\end{equation}
which implies that
\begin{equation}
\sum_{j=1}^N\left|\int_{\R^d}\varphi(x)d(\d_{x_j}-\d_{x_j}^{(\et_j)})(x)\right| \lesssim_d \|\varphi\|_{B_{2,\infty}^0(\R^d)}\sum_{j=1}^N \frac{\et_j}{\al_j^{(2+d)/2}}.
\end{equation}
For the high-frequency term, we use the crude estimate
\begin{equation}
\left|\int_{\R^d} (P_{>|\log_2\al_j|}\varphi)(x)d(\d_{x_j}-\d_{x_j}^{(\et_j)})(x)\right| \leq 2\|P_{>|\log_2\al_j|}\varphi\|_{L^\infty(\R^d)} \lesssim_{d} \sum_{k\geq |\log_2\al_j|} 2^{kd/2} \|P_k\varphi\|_{L^2(\R^d)},
\end{equation}
where the ultimate inequality follows by Bernstein's lemma. Next, write
\begin{equation}
\int_{\R^d}\varphi(x)d(\sum_{j=1}^N\d_{x_j}^{(\et_j)}-N\mu)(x) = -\int_{\R^d}\varphi(x)(\D H_{N,\ueta_N}^{\mu,\ux_N})(x)dx,
\end{equation}
where $H_{N,\ueta_N}^{\mu,\ux_N}$ is defined in \eqref{eq:HN_trun_def}, integrate by parts once, then apply Cauchy-Schwarz to obtain
\begin{equation}
\left|\int_{\R^d}\varphi(x)(\D H_{N,\ueta_N}^{\mu,\ux_N})(x)dx\right| \leq \|\nabla\varphi\|_{L^2(\R^d)} \|\nabla H_{N,\ueta_N}^{\mu,\ux_N}\|_{L^2(\R^d)}.
\end{equation}
Choosing $\ueta_N=\ur_{N,\ep_1}$, we apply \cref{cor:grad_H} to obtain that the right-hand side above is $\lesssim_d$
\begin{equation}
\|\nabla\varphi\|_{L^2(\R^d)}\paren*{\Fr_N(\ux_N,\mu)+ N\ep_1^{-1}+ C_{p,d}\|\mu\|_{L^p(\R^d)}N^2\ep_1^{\frac{2p-d}{p}}}^{1/2}.
\end{equation}
Finally, we complete the proof of \eqref{eq:prop_CE_coer_Bes} by choosing $\al_j=\ep_2$ for every $1\leq j\leq N$ and performing a little bookkeeping.

For the Sobolev convergence, it follows from density of $H^{-s}(\R^d)\subset B_{2,1}^{d/2}(\R^d)$, for $s<-d/2$, and duality that
\begin{equation}
\begin{split}
\|\mu-\frac{1}{N}\sum_{i=1}^N\d_{x_i}\|_{H^s(\R^d)} &\lesssim_{d,p} \paren*{\Fr_N^{avg}(\ux_N,\mu) + \frac{1}{N\ep_1} + \|\mu\|_{L^p(\R^d)}\ep_1^{\frac{2p-d}{p}}}^{1/2} + \frac{\ep_1}{\ep_2^{(2+d)/2}} + \frac{1}{\ep_2^{(d+2s)/2}}.
\end{split}
\end{equation}
Choosing $\ep_2=\ep_1^{1/(1-s)}$ and $\ep_1 = N^{-(2p-d)/(3p-d)}$ yields \eqref{eq:prop_CE_coer_Sob}. The weak-* convergence now follows by a standard approximation argument, using that $H^{-s}(\R^d)$, for $s<-d/2$, is a dense subspace of $C_0(\R^d)$, the Banach space of continuous functions vanishing at infinity. We omit the details.
\end{proof}

\section{Key Proposition}
\label{sec:kprop}
In this section, we prove \cref{prop:kprop} (cf. \cite[Proposition 1.7]{Rosenzweig2020_PVMF}), which is the key ingredient in the proof of \cref{thm:main}.

\begin{restatable}{prop}{krop}
\label{prop:kprop}
Assume that $\mu\in \P(\R^d)\cap L^p(\R^d)$, for some $d<p\leq\infty$. Then for any bounded, log-Lipschitz vector field $v:\R^d\rightarrow\R^d$ and vector $\ux_N\in (\R^d)^N\setminus\D_N$, we have the estimate
\begin{equation}
\label{eq:kprop}
\begin{split}
&\frac{1}{N^2}\left|\int_{(\R^d)^2\setminus\D_2} \paren*{v(x)-v(y)}\cdot(\nabla\g)(x-y)d(\sum_{i=1}^N\d_{x_i}-N\mu)(x)d(\sum_{i=1}^N\d_{x_i}-N\mu)(y)\right| \\
&\lesssim_d \|v\|_{LL(\R^d)}\paren*{\ln_+(N^2H_{N,d}) + \ln(\ep_2^{-1})}\Fr_N^{avg}(\ux_N,\mu) + \frac{\|v\|_{LL(\R^d)}\paren*{\ln(\ep_2^{-1})\tl{\g}(\ep_1)+\ln_+(N^2H_{N,d})\tl{\g}(\ep_3)}}{N} \\
&\ph + C_{p,d}\|v\|_{LL(\R^d)}\|\mu\|_{L^p(\R^d)}\paren*{\ln(\ep_2^{-1})\ep_4^{\frac{2p-d}{p}}+\ln_+(N^2H_{N,d})\ep_3^{\frac{2p-d}{p}}} \\
&\ph + \frac{\ep_1}{\ep_4^{d-1}}\paren*{\frac{\|v\|_{L^\infty(\R^d)}}{\ep_4} + \|v\|_{LL(\R^d)}\ln(\ep_1^{-1})} + \|v\|_{LL(\R^d)}\|\mu\|_{L^p(\R^d)}^{\frac{(d-1)p}{d(p-1)}}\ep_1\ln(\ep_1^{-1}) \\
&\ph + \|v\|_{LL(\R^d)}\ep_2\ln(\ep_2^{-1})\paren*{\frac{1}{\ep_3^{d-1}}+\|\mu\|_{L^p(\R^d)}^{\frac{(d-1)p}{d(p-1)}}} \\
&\ph + \|v\|_{L^\infty(\R^d)}\paren*{\|\mu\|_{L^p(\R^d)}\ep_1^{1-\frac{d}{p}}+\|\mu\|_{L^\infty(\R^d)}\ep_1\ln(\ep_1^{-1})1_{\geq\infty}(p)}
\end{split}
\end{equation}
for all parameters $(\ep_1,\ep_2,\ep_3,\ep_4)\in (\R_+)^4$ satisfying $0<2\ep_1<\ep_2<\ep_3,\ep_4\ll 1$. Here, $\|\cdot\|_{LL(\R^d)}$ denotes the semi-norm defined in \cref{def:LL} and $C_{p,d}, C_{\infty,d}>0$ are constants.
\end{restatable}

As in our previous treatment of the 2D Coulomb case \cite{Rosenzweig2020_PVMF}, our method of proof is inspired by that of \cite[Proposition 2.3]{Serfaty2020}, but requires a more sophisticated analysis, as we no longer have at our disposal the assumption that $\nabla v\in L^\infty(\R^d;(\R^d)^{\otimes 2})$. To make the overall argument modular and easier for the reader to digest, we have divided the proof into several lemmas corresponding to the main steps of the argument, which we recall are
\begin{enumerate}[(S1)]
\item\label{item:moll}
Mollification,
\item\label{item:renorm}
Renormalization,
\item\label{item:diag}
Analysis of Diagonal Terms,
\item\label{item:recomb}
Recombination,
\item\label{item:conc}
Conclusion.
\end{enumerate}
Since we included an extended discussion in our 2D work, we focus here on \ref{item:moll}, as it is the main source of new difficulty, and refer the reader to the beginning of \cite[Section 4.1]{Rosenzweig2020_PVMF}, as well as Serfaty's original article \cite{Serfaty2020}, for more on the modulated energy method.

In \ref{item:moll}, we introduce a parameter $0<\ep_2\ll 1$ and replace the vector field $v$ with the mollified vector field $v_{\ep_2} \coloneqq v\ast\chi_{\ep_2}$, for a standard approximate identity $\chi_{\ep_2}(x)\coloneqq {\ep_2}^{-d}\chi({\ep_2}^{-1}x)$. We now need to estimate the quantity
\begin{equation}
\left|\int_{(\R^d)^2\setminus\D_2}\paren*{v-v_{\ep_2}}(x) \cdot (\nabla\g)(x-y) d(\sum_{i=1}^N\d_{x_i}-N\mu)(x)d(\sum_{i=1}^N\d_{x_i}-N\mu)(y)\right|
\end{equation}
Expanding the integral into four terms, we concentrate on the term
\begin{equation}
\frac{1}{2}\sum_{1\leq i\neq j\leq N} \paren*{(v-v_{\ep_2})(x_i)-(v-v_{\ep_2})(x_j)}\cdot (\nabla\g)(x_i-x_j),
\end{equation}
which we have put in symmetrized form. As in the proof of \cite[Lemma 4.2]{Rosenzweig2020_PVMF}, we want to split the sum into ``close'' and ``far'' particle pairs with threshold $1\gg\ep_3\gg \ep_2$:
\begin{equation}
\sum_{1\leq i\neq j\leq N} = \sum_{{1\leq i\neq j\leq N}\atop {|x_i-x_j|<\ep_3}} + \sum_{{1\leq i\neq j\leq N}\atop {|x_i-x_j|\geq \ep_3}}.
\end{equation}
The pairs of close particles are the problematic contribution. In the 2D Coulomb case, we can use the log-Lipschitz regularity of $v-v_{\ep_2}$ to estimate
\begin{equation}
\sum_{{1\leq i\neq j\leq N}\atop{|x_i-x_j|<\ep_3}} \left|\paren*{(v-v_{\ep_2})(x_i)-(v-v_{\ep_2})(x_j)}\cdot (\nabla\g)(x_i-x_j)\right| \lesssim \|v\|_{LL(\R^d)}\sum_{{1\leq i\neq j\leq N}\atop{|x_i-x_j|<\ep_3}} \g(x_i-x_j),
\end{equation}
and we can then estimate the right-hand side in terms of the modulated energy through the 2D analogue of \cref{lem:count}. However, for the 3D+ Coulomb case, we already saw in the introduction that following the same argument leads to
\begin{equation}
\sum_{{1\leq i\neq j\leq N}\atop{|x_i-x_j|<\ep_3}} \left|\paren*{(v-v_{\ep_2})(x_i)-(v-v_{\ep_2})(x_j)}\cdot (\nabla\g)(x_i-x_j)\right| \lesssim_d \|v\|_{LL(\R^d)}\sum_{{1\leq i\neq j\leq N}\atop{|x_i-x_j|<\ep_3}} \g(x_i-x_j)\ln|x_i-x_j|^{-1},
\end{equation}
which is not amenable to an application of \cref{lem:count}. To get rid of the log factor, we crudely use the minimal separation bound $|x_i-x_j| \gtrsim N^{-2}$ coming from conservation of energy, and then apply \cref{lem:count}. This introduces an inefficiency into the argument, ultimately leading to the short-times restriction.

\subsection{Stress-Energy Tensor}\label{ssec:kprop_SET}
We recall the definition of the stress-energy tensor. For functions $\varphi,\psi\in C_{loc}^1(\R^d)$, we define their \emph{stress-energy tensor} $\{\comm{\varphi}{\psi}_{SE}^{ij}\}_{i,j=1}^d$ to be the $d\times d$ matrix with entries
\begin{equation}
\label{eq:se_ten_def}
\comm{\varphi}{\psi}_{SE}^{ij} \coloneqq \paren*{\p_i\varphi\p_j\psi + \p_j\varphi\p_i\psi} - \d_{ij}\nabla\varphi\cdot\nabla\psi,
\end{equation}
where $\d_{ij}$ is the Kronecker delta function. By a density argument, it follows that the stress-energy tensor is well-defined in $L^2(\R^d; (\R^d)^{\otimes 2})$ for any functions $\varphi,\psi\in \dot{H}^1(\R^d)$. By direct computation, we have the divergence identity
\begin{equation}
\p_i\comm{\varphi}{\psi}_{SE}^{ij} = \D \varphi \p_j\psi + \D\psi\p_j\varphi \qquad j\in\{1,\ldots,d\},
\end{equation}
for any $\varphi,\psi \in \dot{H}^1(\R^d)\cap \dot{H}^2(\R^d)$, where we have used the convention of Einstein summation in index $i$. The following lemma from \cite{Serfaty2020} is used extensively in the sequel. Note the requirement that the vector field be Lipschitz, not log-Lipschitz, which will necessitate a mollification in order to apply the lemma.

\begin{lemma}[{\cite[Lemma 4.3]{Serfaty2020}}]
\label{lem:SE_div}
Let $v\in W^{1,\infty}(\R^d;\R^d)$. For any measures $\mu,\nu\in \M(\R^d)$, such that
\begin{equation}
\int_{\R^d}(|(\nabla\g\ast |\mu|)(x)|^2 + |(\nabla\g\ast|\nu|)(x)|^2)dx<\infty,
\end{equation}
we have the identity
\begin{equation}
\int_{(\R^d)^2} \paren*{v(x)-v(y)}\cdot(\nabla\g)(x-y)d\mu(x)d\nu(y) = \int_{\R^d}(\nabla v)(x) : \comm{\g\ast\mu}{\g\ast\nu}_{SE}(x)dx.
\end{equation}
\end{lemma}

\subsection{Step 1: Mollification}
\label{ssec:kprop_moll}
Let $\chi\in C_c^\infty(\R^d)$ be a radial, nonincreasing bump function satisfying
\begin{equation}
\label{eq:chi}
\int_{\R^d}\chi(x) dx=1, \quad 0\leq \chi\leq 1, \quad \chi(x) = \begin{cases} 1, & {|x|\leq \frac{1}{4}} \\ 0, & {|x|>1} \end{cases}.
\end{equation}
For $\ep>0$, set
\begin{equation}
\label{eq:chi_v_ep}
\chi_\ep(x)\coloneqq \ep^{-d}\chi(x/\ep) \qquad \text{and} \qquad  v_\ep(x) \coloneqq (\chi_\ep\ast v)(x),
\end{equation}
where the convolution $\chi_{\ep}\ast v$ is performed component-wise. Evidently, $v_\vep$ is $C^\infty(\R^d;\R^d)$ and
\begin{equation}
\|\nabla^k v_\ep\|_{L^\infty(\R^d)} \lesssim_{k,d} \ep^{-k}, \qquad \forall k\in\N_0.
\end{equation}
The next lemma controls the error from replacing $v$ with $v_{\ep}$ in the left-hand side of inequality \eqref{eq:kprop} in terms of $\Fr_N(\ux_N,\mu)$.

\begin{lemma}
\label{lem:kprop_error}
There exists a constant $C_{p,d}>0$ such that for every $0<\ep_2,\ep_3\ll 1$, we have the estimate
\begin{equation}
\label{eq:LHS_split}
\begin{split}	
&\left|\int_{(\R^d)^2\setminus\D_2}\paren*{v-v_{\ep_2}}(x) \cdot (\nabla\g)(x-y) d(\sum_{i=1}^N\d_{x_i}-N\mu)(x)d(\sum_{i=1}^N\d_{x_i}-N\mu)(y)\right| \\
&\lesssim_d \|v\|_{LL(\R^d)}\ln_+(N^2H_{N,d})\paren*{\Fr_N(\ux_N,\mu) + N\tl{\g}(\ep_3)+C_{p,d}N^2\|\mu\|_{L^p(\R^d)}{\ep_3}^{(2p-d)/p}} \\
&\ph+N^2\|v\|_{LL(\R^d)}\ep_2|\ln\ep_2|\paren*{\frac{1}{\ep_3^{d-1}}+\|\mu\|_{L^p(\R^d)}^{(d-1)p/d(p-1)}}.
\end{split}
\end{equation}
\end{lemma}
\begin{proof}
We split the left-hand side of \eqref{eq:LHS_split} into a sum of four terms and estimate each separately.
\begin{itemize}[leftmargin=*]
\item
Observe that
\begin{align}
&\left|\int_{(\R^d)^2\setminus\D_2} (v-v_{\ep_2})(x)\cdot(\nabla\g)(x-y)d(\sum_{i=1}^N\d_{x_i})(x)d(\sum_{i=1}^N\d_{x_i})(x)\right| \nn\\
&= \left|\sum_{1\leq i\neq j\leq N} (v-v_{\ep_2})(x_i)\cdot(\nabla\g)(x_i-x_j)\right| \nn\\
&= \left|\frac{1}{2}\sum_{1\leq i\neq j\leq N} \paren*{(v-v_{\ep_2})(x_i)-(v-v_{\ep_2})(x_j)} \cdot(\nabla\g)(x_i-x_j)\right|, \label{eq:ti_cl_far}
\end{align}
where the ultimate equality follows from swapping $i$ and $j$ and using that $(\nabla\g)(x-y)=-(\nabla\g)(y-x)$. For given $\ep_3>0$, the triangle inequality implies
\begin{align}
\label{eq:moll_cl_far}
\eqref{eq:ti_cl_far} &\lesssim \left|\sum_{{1\leq i\neq j\leq N}\atop{|x_i-x_j|\leq \ep_3}}\paren*{(v-v_{\ep_2})(x_i)-(v-v_{\ep_2})(x_j)} \cdot(\nabla\g)(x_i-x_j)\right| \nn\\
&\ph + \left|\sum_{{1\leq i\neq j\leq N}\atop{|x_i-x_j|> \ep_3}}\paren*{(v-v_{\ep_2})(x_i)-(v-v_{\ep_2})(x_j)} \cdot(\nabla\g)(x_i-x_j)\right| \nn\\
&\eqqcolon \mathrm{Term}_1 + \mathrm{Term}_2.
\end{align}
By another application of the triangle inequality, followed by using that $\|v_{\ep_2}\|_{LL(\R^d)}\leq \|v\|_{LL(\R^d)}$, together with the elementary bound
\begin{equation}
|(\nabla\g)(x_i-x_j)| \lesssim_d \frac{1}{|x_i-x_j|^{d-1}},
\end{equation}
we find that
\begin{align}
\mathrm{Term}_1 &\lesssim_d \sum_{{1\leq i\neq j\leq N}\atop{|x_i-x_j|\leq\ep_3}} \|v\|_{LL(\R^d)}\g(x_i-x_j)\ln |x_i-x_j|^{-1}.
\end{align}
The get rid of the log factor, we use the conservation of the $N$-body Hamiltonian $H_{N,d}$ from \eqref{eq:Cou_sys_ham} to obtain the crude minimal separation estimate\footnote{This bound will ultimately be the source of the short-time restriction in \cref{thm:main}.}
\begin{equation}
\min_{1\leq i\neq j\leq N} |x_i-x_j| \gtrsim_d (N^2 H_{N,d})^{-1/(d-2)}.
\end{equation}
Hence,
\begin{align}
&\sum_{{1\leq i\neq j\leq N}\atop{|x_i-x_j|\leq\ep_3}} \|v\|_{LL(\R^d)}\g(x_i-x_j)\ln |x_i-x_j|^{-1} \nn\\
&\lesssim_d \|v\|_{LL(\R^d)}\ln_+(N^2H_{N,d})\sum_{{1\leq i\neq j\leq N}\atop{|x_i-x_j|\leq \ep_3}} \g(x_i-x_j) \nn \\
&\lesssim \|v\|_{LL(\R^d)}\ln_+(N^2H_{N,d})\paren*{\Fr_N(\ux_N,\mu) + N\tl{\g}(2\ep_3) + C_{p,d} N^2\|\mu\|_{L^p(\R^d)}\ep_3^{(2p-d)/p} },
\end{align}
where the ultimate inequality follows by application of \cref{lem:count}.

For $\mathrm{Term}_2$, we observe the bound
\begin{equation}
|(\nabla\g)(x_i-x_j)| \lesssim_d \ep_3^{-d+1},
\end{equation}
which is immediate from the definition of $\nabla\g$ and the separation of $x_i$ and $x_i$, and the bound
\begin{equation}
|(v-v_{\ep_2})(x_i)-(v-v_{\ep_2})(x_j)| \lesssim \|v\|_{LL(\R^d)}\ep_2|\ln\ep_2|,
\end{equation}
which follows from estimate \eqref{eq:v_diff_Linf} of \cref{lem:conv_bnds} and the triangle inequality, in order to obtain that
\begin{equation}
\mathrm{Term}_2 \lesssim_d \sum_{{1\leq i\neq j\leq N}\atop {|x_i-x_j|>\ep_3}} \frac{\|v\|_{LL(\R^d)}\ep_2 |\ln\ep_2|}{\ep_3^{d-1}} \leq \frac{N^2\|v\|_{LL(\R^d)}\ep_2|\ln\ep_2|}{\ep_3^{d-1}}.
\end{equation}
\item
Next, observe that by the Fubini-Tonelli theorem and H\"older's inequality,
\begin{align}
\left|\int_{(\R^d)^2\setminus\D_2} (v-v_{\ep_2})(x)\cdot(\nabla\g)(x-y)d(N\mu)(x)d(N\mu)(y)\right| &= N^2\left|\int_{\R^d} (v-v_{\ep_2})(x)\cdot(\nabla\g\ast\mu)(x)d\mu(x)\right| \nn\\
&\leq N^2\|v-v_{\ep_2}\|_{L^\infty(\R^d)}\|\nabla\g\ast\mu\|_{L^\infty(\R^d)}.
\end{align}
Applying \cref{lem:conv_bnds} to the first factor and \cref{lem:PE_bnds} to the second factor in the ultimate line, we find that
\begin{equation}
\label{eq:L_inf_mu_d}
N^2\|v-v_{\ep_2}\|_{L^\infty(\R^d)}\|\nabla\g\ast\mu\|_{L^\infty(\R^d)} \lesssim_{p,d} N^2\ep_2|\ln\ep_2| \|v\|_{LL(\R^d)} \|\mu\|_{L^p(\R^d)}^{(d-1)p/d(p-1)}.
\end{equation}
\item
Next, observe that by the Fubini-Tonelli theorem,
\begin{align}
\left|\int_{(\R^d)^2\setminus\D_2}(v-v_{\ep_2})(x)\cdot\nabla\g(x-y)d(\sum_{i=1}^N\d_{x_i})(x)d(N\mu)(y)\right| &= N\left|\sum_{i=1}^N (v-v_{\ep_2})(x_i) \cdot (\nabla\g\ast\mu)(x_i)\right| \nn\\
&\lesssim_d N^2\ep_2|\ln\ep_2|\|v\|_{LL(\R^d)} \|\mu\|_{L^{p}(\R^d)}^{(d-1)p/d(p-1)},
\end{align}
again by using \Cref{lem:PE_bnds,lem:conv_bnds}.
\item
Finally, by proceeding similarly as to the previous case,
\begin{align}
\left|\int_{(\R^d)^2\setminus\D_2}(v-v_{\ep_2})(x)\cdot\nabla\g(x-y)d(\sum_{i=1}^N\d_{x_i})(y)d(N\mu)(x)\right| &= N\left|\sum_{i=1}^N(\nabla\g\ast\paren*{(v-v_{\ep_2})\mu})(x_i)\right| \nn\\
&\lesssim N^2\ep_2|\ln\ep_2|\|v\|_{LL(\R^d)} \|\mu\|_{L^{p}(\R^d)}^{(d-1)p/d(p-1)}.
\end{align}
\end{itemize}
\end{proof}

\subsection{Step 2: Renormalization}
\label{ssec:kprop_renorm}
We now proceed to step 2 of the proof, working with the mollified vector field $v_{\ep_2}$. Crucially, this step relies on the qualitative assumption that $v_{\ep_2}$ is Lipschitz.

\begin{lemma}\label{lem:kprop_renorm}
It holds that
\begin{equation}
\label{eq:kprop_S2_id}
\begin{split}
&\int_{(\R^d)^2\setminus\D_2} (v_{\ep_2}(x)-v_{\ep_2}(y))\cdot (\nabla\g)(x-y)d(\sum_{i=1}^N\d_{x_i}-N\mu)(x)d(\sum_{i=1}^N\d_{x_i}-N\mu)(y)\\
&=\lim_{|\ue_N|\rightarrow 0} \paren*{\int_{\R^d} (\nabla v_{\ep_2})(x): \comm{H_{N,\ue_N}^{\mu,\ux_N}}{H_{N,\ue_N}^{\mu,\ux_N}}_{SE}(x)dx - \sum_{i=1}^N\int_{(\R^d)^2} (v_{\ep_2}(x)-v_{\ep_2}(y))\cdot(\nabla\g)(x-y)\d_{x_i}^{(\eta_i)}(x)\d_{x_i}^{(\eta_i)}(y)}.
\end{split}
\end{equation}
\end{lemma}
\begin{proof}
See the proof of \cite[Lemma 4.3]{Rosenzweig2020_PVMF}, which carries over to the $d\geq 3$ case \emph{mutatis mutandis}.
\end{proof}

\subsection{Step 3: Diagonal Terms}
\label{ssec:kprop_diag}
The following lemma from \cite{Serfaty2020}, provides an identity which analyzes how the diagonal contribution (i.e. the second term in the right-hand side) in identity \eqref{eq:kprop_S2_id} varies as the truncation vector $\ue_N$ varies.

\begin{lemma}
\label{lem:kprop_diag}
Let $\ua_N,\ue_N\in (\R_+)^N$, such that $\al_i\geq \et_i$ for every $i\in\{1,\ldots,N\}$. Then for each $i$,
\begin{equation}
\label{eq:diag_zero_id}
\int_{(\R^d)^2}\paren*{v_{\ep_2}(x)-v_{\ep_2}(y)}\cdot(\nabla\g)(x-y)d\d_{x_i}^{(\alpha_i)}(x)d(\d_{x_i}^{(\eta_i)}-\d_{x_i}^{(\alpha_i)})(y) = 0,
\end{equation}
and consequently,
\begin{equation}
\begin{split}
&\int_{(\R^d)^2}\paren*{v_{\ep_2}(x)-v_{\ep_2}(y)}\cdot(\nabla\g)(x-y)\paren*{d\d_{x_i}^{(\eta_i)}(x)d\d_{x_i}^{(\eta_i)}(y)-d\d_{x_i}^{(\alpha_i)}(x)d\d_{x_i}^{(\alpha_i)}(y)} \\
&=\int_{\R^d}(\nabla v_{\ep_2})(x):\comm{\f_{\alpha_i,\eta_i}(\cdot-x_i)}{\f_{\alpha_i,\eta_i}(\cdot-x_i)}_{SE}(x)dx,
\end{split}
\end{equation}
where we recall the definition of $\f_{\al_i,\et_i}$ from \eqref{eq:f_et_al_def}.
\end{lemma}
\begin{proof}
See \cite[Section 4: Steps 2 and 3]{Serfaty2020}.
\end{proof}

\subsection{Step 4: Recombination}
\label{ssec:kprop_recomb}
Combining \cref{lem:kprop_renorm} from step 2 and \cref{lem:kprop_diag} from step 3 shows that for any $\ua_N\in (\R_+)^N$, such that $\al_i>\et_i$ for every $i\in\{1,\ldots,N\}$, it holds that
\begin{equation}
\label{eq:S4_pref}
\begin{split}
&\int_{(\R^d)^2\setminus\D_2} (v_{\ep_2}(x)-v_{\ep_2}(y))\cdot (\nabla\g)(x-y)d(\sum_{i=1}^N\d_{x_i}-N\mu)(x)d(\sum_{i=1}^N\d_{x_i}-N\mu)(y)\\
&\ph-\paren*{\int_{\R^d} (\nabla v_{\ep_2})(x): \comm{H_{N,\ua_N}^{\mu,\ux_N}}{H_{N,\ua_N}^{\mu,\ux_N}}_{SE}(x)dx - \sum_{i=1}^N\int_{(\R^d)^2} (v_{\ep_2}(x)-v_{\ep_2}(y))\cdot(\nabla\g)(x-y)\d_{x_i}^{(\al_i)}(x)\d_{x_i}^{(\al_i)}(y)} \\
&=\lim_{|\ue_N|\rightarrow 0} \Bigg(\int_{\R^d} (\nabla v_{\ep_2})(x): \paren*{\comm{H_{N,\ue_N}^{\mu,\ux_N}}{H_{N,\ue_N}^{\mu,\ux_N}}_{SE}-\comm{H_{N,\ua_N}^{\mu,\ux_N}}{H_{N,\ua_N}^{\mu,\ux_N}}_{SE}}(x)dx \\
&\ph \hspace{25mm} - \sum_{i=1}^N\int_{\R^d}(\nabla v_{\ep_2})(x):\comm{\f_{\alpha_i,\eta_i}}{\f_{\alpha_i,\eta_i}}_{SE}(x-x_i)dx\Bigg)
\end{split}
\end{equation}
In step 4, we analyze how the first term in the third line varies as $|\ua_N|, |\ue_N| \rightarrow 0$. We follow our earlier 2D analysis for the proof of \cite[Lemma 4.5]{Rosenzweig2020_PVMF} as much as possible, but due to the difference between the 2D and 3D+ Coulomb potentials previously encountered in \eqref{eq:Cou_pot_iss}, relying on the log-Lipschitz regularity of $v_{\ep_2}$ will not suffice. Instead, we purchase the Lipschitz regularity of the mollified vector field $v_{\ep_2}$ at the cost of a factor of $\ln\ep_2^{-1}$, allowing us to use the elementary identity
\begin{equation}
|x\cdot\nabla\g(x)| \lesssim_d \g(x).
\end{equation}

\begin{lemma}
\label{lem:kprop_recomb}
Define the parameters
\begin{equation}
\label{eq:r_def}
r_{i,\ep_1} \coloneqq \min\{\frac{1}{4}\min_{{1\leq j\leq N}\atop{i\neq j}}|x_i-x_j|, \ep_1\}, \qquad \ur_{N,\ep_1} \coloneqq (r_{1,\ep_1},\ldots,r_{N,\ep_1}).
\end{equation}
If $\ua_N,\eta_N\in(\R_+)^N$ are such that $\eta_i<\alpha_i\leq r_{i,\ep_1}$, for every $i\in\{1,\ldots,N\}$, then
\begin{equation}
\label{eq:recomb}
\begin{split}
&\int_{\R^d} (\nabla v_{\ep_2})(x) : \paren*{\comm{H_{N,\ul{\eta}_N}^{\mu,\ux_N}}{H_{N,\ul{\eta}_N}^{\mu,\ux_N}}_{SE} - \comm{H_{N,\ua_N}^{\mu,\ux_N}}{H_{N,\ua_N}^{\mu,\ux_N}}_{SE}}(x)dx \\
&= \sum_{i=1}^N \int_{\R^d} (\nabla v_{\ep_2})(x): \comm{\f_{\al_i,\et_i}}{\f_{\al_i,\et_i}}_{SE}(x-x_i)dx  + \mathrm{Error}_{\ep_2,\ua_N,\ue_N},
\end{split}
\end{equation}
where there exist constants $C_{p,d},C_{\infty,d}>0$ such that
\begin{equation}
\label{eq:recomb_err_bnd}
\begin{split}
\left|\mathrm{Error}_{\ep_2,\ua_N,\ue_N}\right| &\lesssim_{d} \|v\|_{LL(\R^d)}\ln(\ep_2^{-1}) \paren*{\Fr_N(\ux_N,\mu) + N\tl{\g}(\ep_4) + C_{p,d} N^2 \|\mu\|_{L^p(\R^d)} {\ep_4}^{\frac{(2p-d)}{p}}} \\
&\ph + N\sum_{i=1}^N \al_i\paren*{\frac{\|v\|_{L^\infty(\R^d)}}{{\ep_4}^d} +\frac{\|v\|_{{LL}(\R^d)} |\ln \al_i|}{\ep_4^{d-1}}} +N\|v\|_{{LL}(\R^d)} \|\mu\|_{L^p(\R^d)}^{\frac{(d-1)p}{d(p-1)}}\sum_{i=1}^N \al_i\ln(\al_i^{-1}) \\
&\ph + N\|v\|_{L^\infty(\R^d)}\sum_{i=1}^N \paren*{C_{p,d}\|\mu\|_{L^p(\R^d)}\al_i^{1-\frac{d}{p}} + \|\mu\|_{L^\infty(\R^d)}\al_i\ln (\al_i^{-1})1_{\geq \infty}(p)}.
\end{split}
\end{equation}
\end{lemma}
\begin{proof}
It follows from identity \eqref{eq:g_conv_smear} that
\begin{align}
H_{N,\ua_N}^{\mu,\ux_N} &= H_{N,\ue_N}^{\mu,\ux_N} +\sum_{i=1}^N \f_{\alpha_i,\eta_i}(\cdot-x_i).
\end{align}
Observe that $\g_{\eta_i}(\cdot-x_i) = \g_{\alpha_i}(\cdot-x_i)$ outside the ball $B(x_i,\alpha_i)$ (recall that $\al_i>\et_i$ by assumption). Also observe that the closed balls $\ol{B}(x_i,\alpha_i)$ are pairwise disjoint. To see this property, observe that
\begin{equation}
\al_i \leq r_{i,\ep_1} \leq \frac{|x_i-x_j|}{4} \quad \text{and} \quad \al_j \leq r_{j,\ep_1} \leq \frac{|x_i-x_j|}{4}, \qquad \forall i\neq j,
\end{equation}
by definition of $r_{i,\ep_1},r_{j,\ep_1}$ and requirement that $\al_i\leq r_{i,\ep_1}$ and $\al_j\leq r_{j,\ep_1}$. Thus, we find that
\begin{equation}
H_{N,\ua_N}^{\mu,\ux_N}(x) = H_{N,\ue_N}^{\mu,\ux_N}(x), \qquad \forall x\in \R^d\setminus\bigcup_{i=1}^N B(x_i,\alpha_i).
\end{equation}
So by the bilinearity of the stress-energy tensor, we have the point-wise a.e. identity
\begin{equation}
\label{eq:se_diff_id}
\begin{split}
\comm{H_{N,\ua_N}^{\mu,\ux_N}}{H_{N,\ua_N}^{\mu,\ux_N}}_{SE} - \comm{H_{N,\ue_N}^{\mu,\ux_N}}{H_{N,\ue_N}^{\mu,\ux_N}}_{SE} &= \sum_{i=1}^N \comm{\f_{\alpha_i,\eta_i}}{\f_{\alpha_i,\eta_i}}_{SE}1_{B(x_i,\alpha_i)} \\
&\ph + 2\sum_{i=1}^N\comm{\f_{\alpha_i,\eta_i}}{H_{N,\ue_N}^{\mu,\ux_N}}_{SE}1_{B(x_i,\alpha_i)},
\end{split}
\end{equation}
where we use that the terms $\comm{\f_{\alpha_i,\eta_i}}{\f_{\alpha_j,\eta_j}}_{SE}$, for $i\neq j$, vanish due to $\ol{B}(x_i,\alpha_i)\cap \ol{B}(x_j,\alpha_j)=\emptyset$ and similarly, $\comm{\f_{\alpha_j,\eta_j}}{H_{N,\ue_N}^{\mu,\ux_N}}_{SE}1_{B(x_i,\alpha_i)}\equiv 0$. Thus, using identity \eqref{eq:se_diff_id}, we see that
\begin{align}
&\int_{\R^d}(\nabla v_{\ep_2})(x) : \paren*{\comm{H_{N,\ua_N}^{\mu,\ux_N}}{H_{N,\ua_N}^{\mu,\ux_N}}_{SE} - \comm{H_{N,\ue_N}^{\mu,\ux_N}}{H_{N,\ue_N}^{\mu,\ux_N}}_{SE}}(x)dx \nn\\
&=\sum_{i=1}^N \int_{B(x_i,\alpha_i)} (\nabla v_{\ep_2})(x): \paren*{\comm{\f_{\alpha_i,\eta_i}}{\f_{\alpha_i,\eta_i}}_{SE}+2\comm{\f_{\alpha_i,\eta_i}}{H_{N,\ue_N}^{\mu,\ux_N}}_{SE}}(x)dx \nn\\
&=\sum_{i=1}^N\int_{\R^d}(\nabla v_{\ep_2})(x) :\paren*{\comm{\f_{\alpha_i,\eta_i}}{\f_{\alpha_i,\eta_i}}_{SE}+2\comm{\f_{\alpha_i,\eta_i}}{H_{N,\ue_N}^{\mu,\ux_N}}_{SE}}(x)dx,  \label{eq:recomb_nte}
\end{align}
where the ultimate line follows from the fact that $\nabla\f_{\alpha_i,\eta_i}$ vanishes outside $\ol{B}(x_i,\alpha_i)$. Comparing this last expression to the right-hand side of identity \eqref{eq:recomb}, we see that we only need to estimate the modulus of
\begin{equation}
\label{eq:recomb_split_pre}
\begin{split}
&2\sum_{i=1}^N \int_{\R^d}(\nabla v_{\ep_2})(x) : \comm{\f_{\alpha_i,\eta_i}}{H_{N,\ue_N}^{\mu,\ux_N}}_{SE}(x)dx\\
&=2\sum_{i=1}^N \int_{(\R^d)^2}\paren*{v_{\ep_2}(x)-v_{\ep_2}(y)}\cdot (\nabla\g)(x-y)d(\sum_{j=1}^N\d_{x_j}^{(\alpha_j)}-N\mu)(x)d(\d_{x_i}^{(\eta_i)}-\d_{x_i}^{(\alpha_i)})(y),
\end{split}
\end{equation}
where the right-hand side follows from an application of \cref{lem:SE_div}. Note that by identity \eqref{eq:diag_zero_id},
\begin{equation}
\label{eq:recomb_split}
\eqref{eq:recomb_split_pre} = 2\sum_{i=1}^N \int_{(\R^d)^2}\paren*{v_{\ep_2}(x)-v_{\ep_2}(y)}\cdot (\nabla\g)(x-y)d(\sum_{{1\leq j\leq N} \atop {j\neq i}}\d_{x_j}^{(\alpha_j)}-N\mu)(x)d(\d_{x_i}^{(\eta_i)}-\d_{x_i}^{(\alpha_i)})(y),
\end{equation}
which splits into a sum of two terms defined below:
\begin{align}
\mathrm{Term}_1 &\coloneqq 2\sum_{i=1}^N \int_{(\R^d)^2}v_{\ep_2}(x)\cdot(\nabla\g)(x-y)d(\sum_{{1\leq j\leq N}\atop {j\neq i}}\d_{x_j}^{(\alpha_j)}-N\mu)(x)d(\d_{x_i}^{(\eta_i)}-\d_{x_i}^{(\alpha_i)})(y),  \label{eq:recomb_T1_def}\\
\mathrm{Term}_2 &\coloneqq -2\sum_{i=1}^N\int_{(\R^d)^2} v_{\ep_2}(y)\cdot(\nabla\g)(x-y)d(\sum_{{1\leq j\leq N}\atop {j\neq i}}\d_{x_j}^{(\alpha_j)}-N\mu)(x)d(\d_{x_i}^{(\eta_i)}-\d_{x_i}^{(\alpha_i)})(y). \label{eq:recomb_T2_def}
\end{align}

\begin{description}[leftmargin=*]
\item[Estimate for $\mathrm{Term}_1$]
We observe that by using the Fubini-Tonelli theorem to first integrate out the $y$-variable and then applying identity \eqref{eq:f_conv_smear},
\begin{equation}
\mathrm{Term}_1 = 2\sum_{i=1}^N\int_{\R^d}v_{\ep_2}(x)\cdot (\nabla\f_{\alpha_i,\eta_i})(x-x_i)d(\sum_{{1\leq j\leq N}\atop {j\neq i}}\d_{x_j}^{(\alpha_j)}-N\mu)(x).
\end{equation}
Since $\supp(\nabla\f_{\alpha_i,\eta_i}(\cdot-x_i))\subset \ol{B}(x_i,\alpha_i)$ and $\supp(\d_{x_j}^{(\alpha_j)})\subset \ol{B}(x_j,\alpha_j)$, which is disjoint from $\ol{B}(x_i,\alpha_i)$, for $j\neq i$, we observe the cancellation
\begin{equation}
\mathrm{Term}_1 = -2N\sum_{i=1}^N\int_{\R^d}v_{\ep_2}(x)\cdot(\nabla\f_{\alpha_i,\eta_i})(x)d\mu(x).
\end{equation}
By triangle and H\"older's inequalities
\begin{equation}
\label{eq:T1_Hold}
\left|-2N\sum_{i=1}^N\int_{\R^d}v_{\ep_2}(x)\cdot(\nabla\f_{\alpha_i,\eta_i})(x)d\mu(x)\right| \lesssim N\sum_{i=1}^N \|v_{\ep_2}\|_{L^{p_1}(\R^d)} \|\nabla\f_{\alpha_i,\eta_i}\|_{L^{p_2}(\R^d)} \|\mu\|_{L^{p_3}(\R^d)},
\end{equation}
where $1\leq p_1,p_2,p_3\leq \infty$ are H\"older conjugate. We choose $p_1=\infty$ and $p_3=p$ and apply \cref{lem:f_et_al_bnds} to $\|\nabla\f_{\alpha_i,\eta_i}\|_{L^{p_2}(\R^d)}$, obtaining
\begin{equation}
\eqref{eq:T1_Hold} \lesssim_{p,d} N\sum_{i=1}^N\alpha_i^{(p-d)/p} \|v_{\ep_2}\|_{L^\infty(\R^d)} \|\mu\|_{L^{p}(\R^d)}.
\end{equation}
Using that $\|v_{\ep_2}\|_{L^\infty(\R^d)}\leq \|v\|_{L^\infty(\R^d)}$ by \cref{lem:conv_bnds}, we conclude that
\begin{equation}
\label{eq:recomb_T1_fin}
|\mathrm{Term}_1| \lesssim_{p,d} N\|v\|_{L^\infty(\R^d)}\|\mu\|_{L^p(\R^d)}\sum_{i=1}^N\al_i^{(p-d)/p},
\end{equation}
which completes the analysis for $\mathrm{Term}_1$.

\item[Estimate for $\mathrm{Term}_2$]
For $\mathrm{Term}_2$ \eqref{eq:recomb_T2_def}, we use Fubini-Tonelli to first integrate out the $x$-variable together with identity \eqref{eq:g_conv_smear}, finding
\begin{equation}
\begin{split}
\mathrm{Term}_2 &= 2\sum_{1\leq i\neq j\leq N}\int_{\R^d} v_{\ep_2}(y)\cdot(\nabla\g_{\alpha_j})(y-x_j)d(\d_{x_i}^{(\eta_i)}-\d_{x_i}^{(\alpha_i)})(y) \\
&\ph - 2\sum_{1\leq i\neq j\leq N}\int_{\R^d}v_{\ep_2}(y)\cdot(\nabla\g\ast\mu)(y)d(\d_{x_i}^{(\eta_i)}-\d_{x_i}^{(\alpha_i)})(y).
\end{split}
\end{equation}
Adding and subtracting $v_{\ep_2}(x_i)$ and $v_{\ep_2}(x_j)$, we can write
\begin{equation}
\begin{split}
&2\sum_{1\leq i\neq j\leq N}\int_{\R^d} v_{\ep_2}(y)\cdot(\nabla\g_{\alpha_j})(y-x_j)d(\d_{x_i}^{(\eta_i)}-\d_{x_i}^{(\alpha_i)})(y) \\
&= \underbrace{2\sum_{1\leq i\neq j\leq N}\int_{\R^d} \paren*{v_{\ep_2}(y)-v_{\ep_2}(x_j)}\cdot(\nabla\g_{\alpha_j})(y-x_j)d(\d_{x_i}^{(\eta_i)}-\d_{x_i}^{(\alpha_i)})(y)}_{\eqqcolon \mathrm{Term}_{2,1}} \\
&\ph + \underbrace{2\sum_{1\leq i\neq j\leq N}\int_{\R^d}v_{\ep_2}(x_j)\cdot(\nabla\g_{\alpha_j})(y-x_j)d(\d_{x_i}^{(\eta_i)}-\d_{x_i}^{(\alpha_i)})(y)}_{\eqqcolon \mathrm{Term}_{2,2}}
\end{split}
\end{equation}
and
\begin{equation}
\begin{split}
&2\sum_{1\leq i\neq j\leq N}\int_{\R^d}v_{\ep_2}(y)\cdot(\nabla\g\ast\mu)(y)d(\d_{x_i}^{(\eta_i)}-\d_{x_i}^{(\alpha_i)})(y)\\
&= \underbrace{2\sum_{1\leq i\neq j\leq N}\int_{\R^d}\paren*{v_{\ep_2}(y)-v_{\ep_2}(x_i)}\cdot(\nabla\g\ast\mu)(y)d(\d_{x_i}^{(\eta_i)}-\d_{x_i}^{(\alpha_i)})(y)}_{\eqqcolon\mathrm{Term}_{2,3}}\\
&\ph +\underbrace{2\sum_{1\leq i\neq j\leq N}\int_{\R^d}v_{\ep_2}(x_i)\cdot(\nabla\g\ast\mu)(y)d(\d_{x_i}^{(\eta_i)}-\d_{x_i}^{(\alpha_i)})(y)}_{\eqqcolon\mathrm{Term}_{2,4}},
\end{split}
\end{equation}
so that
\begin{equation}
\mathrm{Term}_2 = \mathrm{Term}_{2,1} + \mathrm{Term}_{2,2} - \mathrm{Term}_{2,3} - \mathrm{Term}_{2,4}.
\end{equation}
We now estimate $\mathrm{Term}_{2,1},\ldots,\mathrm{Term}_{2,4}$ individually, as itemized below.
\begin{description}[leftmargin=*]
\item[Estimate for $\mathrm{Term}_{2,1}$]
We split the sum $\sum_{1\leq i\neq j\leq N}$ into ``close'' and ``far'' particle pairs $(i,j)$ with distance threshold $\ep_4$:
\begin{align}
\mathrm{Term}_{2,1} &= 2\sum_{{1\leq i\neq j\leq N}\atop {|x_i-x_j|\leq\ep_4}}\int_{\R^d} \paren*{v_{\ep_2}(y)-v_{\ep_2}(x_j)}\cdot(\nabla\g_{\alpha_j})(y-x_j)d(\d_{x_i}^{(\eta_i)}-\d_{x_i}^{(\alpha_i)})(y)\nn\\
&\ph + 2\sum_{{1\leq i\neq j\leq N}\atop{|x_i-x_j|>\ep_4}}\int_{\R^d} \paren*{v_{\ep_2}(y)-v_{\ep_2}(x_j)}\cdot(\nabla\g_{\alpha_j})(y-x_j)d(\d_{x_i}^{(\eta_i)}-\d_{x_i}^{(\alpha_i)})(y) \nn\\
&\eqqcolon \mathrm{Term}_{2,1,1} + \mathrm{Term}_{2,1,2}.
\end{align}

For $\mathrm{Term}_{2,1,1}$, we use that $\|\nabla v_{\ep_2}\|_{L^\infty(\R^d)}\leq \|v\|_{LL(\R^d)}\ln \ep_2^{-1}$ by \cref{lem:conv_bnds} in order to obtain that
\begin{align}
|\mathrm{Term}_{2,1,1}| &\leq \|v\|_{LL(\R^d)}\ln(\ep_2^{-1})\sum_{{1\leq i\neq j\leq N}\atop{|x_i-x_j|\leq\ep_4}}\int_{\R^d} |y-x_j| |(\nabla\g_{\alpha_j})(y-x_j)|d(\d_{x_i}^{(\et_i)}+\d_{x_i}^{(\al_i)})(y) \nn\\
&\lesssim_d \|v\|_{LL(\R^d)}\ln(\ep_2^{-1})\sum_{{1\leq i\neq j\leq N}\atop {|x_i-x_j|\leq\ep_4}}\int_{\R^d}\g(y-x_j)d(\d_{x_i}^{(\eta_i)}+\d_{x_i}^{(\al_i)})(y), \label{eq:T_211_pre}
\end{align}
where the ultimate inequality follows from using the bound
\begin{equation}
|(\nabla\g_{\alpha_j})(y-x_j)| \lesssim_d \frac{1}{|y-x_j|^{d-1}}.
\end{equation}
Since for $i\neq j$, we have that $|x_i-x_j|\geq 4r_{i,\ep_1}$, by definition of $r_{i,\ep_1}$, and that $r_{i,\ep_1}\geq \alpha_i>\eta_i$, by assumption, the reverse triangle inequality and the decreasing property of $\g$ imply that
\begin{equation}
\g(y-x_j) \leq \tl{\g}(|x_i-x_j| - \frac{1}{4}|x_i-x_j|) \lesssim_d \g(x_i-x_j), \qquad \forall y\in \p B(x_i,\alpha_i)\cup\p B(x_i,\eta_i).
\end{equation}
Hence,
\begin{equation}
\eqref{eq:T_211_pre} \lesssim_d \|v\|_{LL(\R^d)}\ln(\ep_2^{-1})\sum_{{1\leq i\neq j\leq N}\atop{|x_i-x_j|\leq\ep_4}} \g(x_i-x_j).
\end{equation}
Applying \cref{lem:count}, we conclude that
\begin{equation}
\label{eq:T_211_fin}
|\mathrm{Term}_{2,1,1}|\lesssim_d \|v\|_{LL(\R^d)} \ln(\ep_2^{-1}) \paren*{\Fr_N(\ux_N,\mu) + N\tl{\g}(\ep_4) + C_{p,d} N^2 \|\mu\|_{L^p(\R^d)} {\ep_4}^{(2p-d)/p}}.
\end{equation}

For $\mathrm{Term}_{2,1,2}$, we see from a change of variable that
\begin{align}
&\int_{\R^d}(v_{\ep_2}(y)-v_{\ep_2}(x_j)) \cdot(\nabla\g_{\alpha_j})(y-x_j)d(\d_{x_i}^{(\eta_i)}-\d_{x_i}^{(\al_i)})(y) \nn\\
&=\int_{\R^d} \paren*{v_{\ep_2}(\eta_i y+x_i-x_j)\cdot(\nabla\g_{\alpha_j})(\eta_i y +x_i-x_j) - v_{\ep_2}(\al_i y+x_i-x_j) \cdot(\nabla\g_{\alpha_j})(\al_i y +x_i-x_j)} d\d_0^{(1)}(y). \label{eq:LL_MVT_bnd}
\end{align}
Since $\ep_4,\ep_3>\ep_2\geq 2\ep_1$ by assumption and since $\eta_i\leq \alpha_i\leq \ep_1$, for every $i\in\{1,\ldots,N\}$, also by assumption, it follows from the log-Lipschitz property of $v_{\ep_2}$ and the (reverse) triangle inequality that
\begin{align}
|\eqref{eq:LL_MVT_bnd}| &\lesssim_d \int_{\R^d} \frac{\|v_{\ep_2}\|_{{LL}(\R^d)}|\et_i-\al_i| \ln|\et_i-\al_i|^{-1}}{|\eta_i y+ x_i-x_j|^{d-1}}d\d_0^{(1)}(y) \nn\\
&\ph + \int_{\R^d}\frac{|v_{\ep_2}(\al_i y + x_i-x_j)| \cdot |\et_i-\al_i| |x_i-x_j|^{d-2}}{(|x_i-x_j|-\al_i)^{d-1}(|x_i-x_j|-\et_i)^{d-1}}d\d_0^{(1)}(y) \nn\\
&\lesssim_d \frac{\|v\|_{{LL}(\R^d)}|\eta_i-\al_i|\ln|\eta_i-\al_i|^{-1}}{\ep_4^{d-1}} + \frac{\|v\|_{L^\infty(\R^d)}|\et_i-\al_i|}{{\ep_4}^{d}}.
\end{align}
Hence,
\begin{align}
|\mathrm{Term}_{2,1,2}| &\lesssim_d \sum_{{1\leq i\neq j\leq N}\atop{|x_i-x_j|>\ep_4}}\Bigg(\frac{\|v\|_{LL(\R^d)}|\eta_i-\al_i| \ln|\eta_i-\al_i|^{-1}}{\ep_4^{d-1}} + \frac{\|v\|_{L^\infty(\R^d)}|\et_i-\al_i|}{{\ep_4}^d}\Bigg) \nn\\
&\lesssim N \sum_{i=1}^N|\et_i-\al_i|\paren*{\frac{\|v\|_{L^\infty(\R^d)}}{\ep_4^d} + \frac{\|v\|_{LL(\R^d)}\ln|\et_i-\al_i|^{-1}}{\ep_4^{d-1}}}.
\end{align}
Since $x\mapsto |x|\ln|x|^{-1}$ is increasing in the region $0<|x|<e^{-1}$, we conclude that
\begin{equation}
\label{eq:T_212_fin}
|\mathrm{Term}_{2,1,2}| \lesssim N\sum_{i=1}^N\al_i \paren*{\frac{\|v\|_{L^\infty(\R^d)}}{\ep_4^d} + \frac{\|v\|_{LL(\R^d)} |\ln \al_i|}{\ep_4^{d-1}}}.
\end{equation}

Combining estimates \eqref{eq:T_211_fin} and \eqref{eq:T_212_fin} for $\mathrm{Term}_{2,1,1}$ and $\mathrm{Term}_{2,1,2}$, respectively, we conclude that
\begin{equation}
\label{eq:T_21_fin}
\begin{split}
|\mathrm{Term}_{2,1}| &\lesssim_d \|v\|_{LL(\R^d)} \ln(\ep_2^{-1}) \paren*{\Fr_N(\ux_N,\mu) + N\tl{\g}(\ep_4) + C_{p,d} N^2 \|\mu\|_{L^p(\R^d)} {\ep_4}^{(2p-d)/p}} \\
&\ph+N\sum_{i=1}^N \al_i\paren*{\frac{\|v\|_{L^\infty(\R^d)}}{\ep_4^d} + \frac{\|v\|_{LL(\R^d)} \ln(\al_i^{-1})}{\ep_4^{d-1}}}.
\end{split}
\end{equation}

\item[Estimate for $\mathrm{Term}_{2,2}$]
We claim that $\mathrm{Term}_{2,2}=0$. To see this, note by identity \eqref{eq:f_conv_smear} that
\begin{equation}
\d_{x_i}^{(\eta_i)}-\d_{x_i}^{(\alpha_i)} = (-\D\f_{\alpha_i,\eta_i})(\cdot-x_i).
\end{equation}
So, we can integrate by parts once, using identity \eqref{eq:delta_smear} and a change of variable, in order to obtain
\begin{equation}
\mathrm{Term}_{2,2} = 2\sum_{1\leq i\neq j\leq N}\int_{\R^d}v_{\ep_2}(x_i)\cdot (\nabla\f_{\alpha_i,\eta_i})(y-x_i)d\d_{x_j}^{(\alpha_j)}(y).
\end{equation}
Now note that $\supp(\nabla\f_{\al_i,\et_i}(\cdot-x_i))\subset \ol{B}(x_i,\al_i)$ by identity \eqref{eq:f_grad_et_al_id} and $\supp(\d_{x_j}^{(\al_j)})\subset \ol{B}(x_j,\al_j)$, while the balls $B(x_i,\al_i), B(x_j,\al_j)$ are disjoint, for $i\neq j$. Thus, we conclude that each of the summands in $\mathrm{Term}_{2,2}$ vanishes, implying
\begin{equation}
\label{eq:T_22_fin}
\mathrm{Term}_{2,2} = 0.
\end{equation}

\item[Estimate for $\mathrm{Term}_{2,3}$]
Using that $\|v_{\ep_2}\|_{{LL}(\R^d)}\leq \|v\|_{{LL}(\R^d)}$ by \cref{lem:conv_bnds}, we find that
\begin{equation}
|\mathrm{Term}_{2,3}| \lesssim_d \|v\|_{{LL}(\R^d)}\sum_{1\leq i\neq j\leq N} \int_{\R^d}|x-x_i| \ln(|x-x_i|^{-1}) |(\nabla\g\ast\mu)(x)| d(\d_{x_i}^{(\eta_i)}+\d_{x_i}^{(\al_i)})(x).
\end{equation}
Since $\|\mu\|_{L^1(\R^d)}=1$, \cref{lem:PE_bnds} implies
\begin{equation}
\|\nabla\g\ast\mu\|_{L^\infty(\R^d)} \lesssim_{p,d} \|\mu\|_{L^p(\R^d)}^{\frac{(d-1)p}{d(p-1)}}.
\end{equation}
Since $x\mapsto |x| \ln|x|^{-1}$ is increasing for $0<|x|<e^{-1}$ and $\eta_i<\al_i$, it follows from the supports of $\d_{x_i}^{(\al_i)}, \d_{x_i}^{(\eta_i)}$ that
\begin{align}
|\mathrm{Term}_{2,3}| &\lesssim_{p,d} \|v\|_{{LL}(\R^d)}\|\mu\|_{L^p(\R^d)}^{\frac{(d-1)p}{d(p-1)}}\sum_{1\leq i\neq j\leq N}\al_i \ln(\al_i^{-1}) \leq N\|v\|_{{LL}(\R^d)} \|\mu\|_{L^p(\R^d)}^{\frac{(d-1)p}{d(p-1)}}\sum_{i=1}^N \al_i \ln(\al_i^{-1}), \label{eq:T_23_fin}
\end{align}
which completes the analysis for $\mathrm{Term}_{2,3}$.

\item[Estimate for $\mathrm{Term}_{2,4}$]
By a change of variable,
\begin{equation}
\label{eq:mod_cont_app}
\begin{split}
&\int_{\R^d}v_{\ep_2}(x_i)\cdot(\nabla\g\ast\mu)(x) d(\d_{x_i}^{(\et_i)}-\d_{x_i}^{(\al_i)})(x) \\
&=\int_{\R^d}v_{\ep_2}(x_i)\cdot\paren*{(\nabla\g\ast\mu)(x_i+\et_iy)-(\nabla\g\ast\mu)(x_i+\al_iy)}d\d_0^{(1)}(y),
\end{split}
\end{equation}
for every $i\in\{1,\ldots,N\}$. By \cref{lem:PE_bnds}, we have the modulus of continuity estimates
\begin{align}
|(\nabla\g\ast\mu)(x_i+\et_iy)-(\nabla\g\ast\mu)(x_i+\al_iy)| &\lesssim \|(\nabla\g\ast\mu)\|_{\dot{C}^{1-\frac{d}{p}}(\R^d) }|\et_i y - \al_i y|^{1-\frac{d}{p}} \nn\\
&\lesssim_{p,d} \|\mu\|_{L^p(\R^d)} |\et_i-\al_i|^{1-\frac{d}{p}}, \qquad \forall |y|=1, \ d<p<\infty
\end{align}
and
\begin{align}
|(\nabla\g\ast\mu)(x_i+\et_iy)-(\nabla\g\ast\mu)(x_i+\al_iy)| &\lesssim \|(\nabla\g\ast\mu)\|_{{LL}(\R^d)} |\et_i y-\al_i y|\ln|\et_iy-\al_iy|^{-1} \nn\\
&\lesssim_d \|\mu\|_{L^\infty(\R^d)} |\et_i-\al_i| \ln|\et_i-\al_i|^{-1}, \qquad \forall |y|=1, \ p=\infty.
\end{align}
Applying these estimates and using that the functions $x\mapsto |x|^{1-\frac{d}{p}}$ and $x\mapsto |x|\ln|x|^{-1}$ are increasing for $0<|x|<e^{-1}$, we find that
\begin{align}
|\eqref{eq:mod_cont_app}| &\leq \begin{cases} C_{p,d}\|v\|_{L^\infty(\R^d)}\|\mu\|_{L^p(\R^d)}\al_i^{1-\frac{d}{p}}, & {d<p<\infty} \\
C_{\infty,d}\|v\|_{L^\infty(\R^d)}\|\mu\|_{L^\infty(\R^d)}\al_i \ln(\al_i^{-1}), & {p=\infty}
\end{cases},
\end{align}
where $C_{p,d}, C_{\infty,d}>0$ are absolute constants. Hence,
\begin{equation}
\label{eq:T_24_fin}
|\mathrm{Term}_{2,4}| \leq \begin{cases} NC_{p,d}\|v\|_{L^\infty(\R^d)}\|\mu\|_{L^p(\R^d)}\sum_{i=1}^N \al_i^{1-\frac{d}{p}}, & {d<p<\infty} \\
N C_{\infty,d}\|v\|_{L^\infty(\R^d)}\|\mu\|_{L^\infty(\R^d)}\sum_{i=1}^N \al_i \ln(\al_i^{-1}), & {p=\infty}
\end{cases},
\end{equation}
which completes the analysis for $\mathrm{Term}_{2,4}$.
\end{description}
Combining estimates \eqref{eq:T_21_fin}, \eqref{eq:T_22_fin}, \eqref{eq:T_23_fin}, and \eqref{eq:T_24_fin} for $\mathrm{Term}_{2,1}$, $\mathrm{Term}_{2,2}$, $\mathrm{Term}_{2,3}$, and $\mathrm{Term}_{2,4}$, respectively, we see that
\begin{equation}
\label{eq:recomb_T2_fin}
\begin{split}
|\mathrm{Term}_2| &\lesssim_d  \|v\|_{LL(\R^d)} \ln(\ep_2^{-1})\paren*{\Fr_N(\ux_N,\mu) + N\tl{\g}(\ep_4) + C_{p,d} N^2 \|\mu\|_{L^p(\R^d)} {\ep_4}^{\frac{2p-d}{p}}}\\
&\ph +N\sum_{i=1}^N \al_i \paren*{\frac{\|v\|_{L^\infty(\R^d)}}{{\ep_4}^d} +\frac{\|v\|_{{LL}(\R^d)} \ln(\al_i^{-1})}{\ep_4^{d-1}}}+N\|v\|_{{LL}(\R^d)} \|\mu\|_{L^p(\R^d)}^{\frac{(d-1)p}{d(p-1)}}\sum_{i=1}^N \al_i\ln(\al_i^{-1}) \\
&\ph + N\|v\|_{L^\infty(\R^d)}\sum_{i=1}^N \paren*{C_{p,d}\|\mu\|_{L^p(\R^d)}\al_i^{1-\frac{d}{p}}1_{d<\cdot<\infty}(p) + C_{\infty,d}\|\mu\|_{L^\infty(\R^d)}\al_i \ln(\al_i^{-1})1_{\geq \infty}(p)}.
\end{split}
\end{equation}
This completes the analysis for $\mathrm{Term}_2$.
\end{description}

With estimate \eqref{eq:recomb_T1_fin} for $\mathrm{Term}_1$, the definition of which we recall from \eqref{eq:recomb_T1_def}, and estimate \eqref{eq:recomb_T2_fin} for $\mathrm{Term}_2$, the definition of which we recall from \eqref{eq:recomb_T2_def}, respectively, together with our starting identity \eqref{eq:recomb_nte}, we conclude that
\begin{equation}
\begin{split}
&\left|\int_{\R^d} (\nabla v_{\ep_2})(x) : \paren*{\comm{H_{N,\ua_N}^{\mu,\ux_N}}{H_{N,\ua_N}^{\mu,\ux_N}}_{SE} - \comm{H_{N,\ue_N}^{\mu,\ux_N}}{H_{N,\ue_N}^{\mu,\ux_N}}_{SE} - \sum_{i=1}^N\comm{\f_{\al_i,\et_i}}{\f_{\al_i,\et_i}}_{SE}(\cdot-x_i)}(x)dx\right| \\
&\lesssim_d   \|v\|_{LL(\R^d)}\ln(\ep_2^{-1}) \paren*{\Fr_N(\ux_N,\mu) + N\tl{\g}(\ep_4) + C_{p,d} N^2 \|\mu\|_{L^p(\R^d)} {\ep_4}^{\frac{(2p-d)}{p}}} \\
&\ph + N\sum_{i=1}^N \al_i\paren*{\frac{\|v\|_{L^\infty(\R^d)}}{{\ep_4}^d} +\frac{\|v\|_{{LL}(\R^d)} |\ln \al_i|}{\ep_4^{d-1}}} +N\|v\|_{{LL}(\R^d)} \|\mu\|_{L^p(\R^d)}^{\frac{(d-1)p}{d(p-1)}}\sum_{i=1}^N \al_i\ln(\al_i^{-1}) \\
&\ph + N\|v\|_{L^\infty(\R^d)}\sum_{i=1}^N \paren*{C_{p,d}\|\mu\|_{L^p(\R^d)}\al_i^{1-\frac{d}{p}} + \|\mu\|_{L^\infty(\R^d)}\al_i\ln(\al_i^{-1})1_{\geq \infty}(p)}.
\end{split}
\end{equation}
Recalling the statement of \cref{lem:kprop_recomb}, we see that the proof of the lemma is complete.
\end{proof}

\subsection{Step 5: Conclusion}
\label{ssec:kprop_conc}
Applying \cref{lem:kprop_recomb} to the identity \eqref{eq:S4_pref}, we have shown that for any $\ua_N\in (\R_+)^N$ with $\alpha_i\leq r_{i,\ep_1}$ for each $i\in\{1,\ldots,N\}$, it holds that
\begin{equation}
\label{eq:conc_RHS}
\begin{split}
&\int_{(\R^d)^2\setminus\D_2} \paren*{v_{\ep_2}(x)-v_{\ep_2}(y)}\cdot(\nabla\g)(x-y)d(\sum_{i=1}^N\d_{x_i}-N\mu)(x)d(\sum_{i=1}^N\d_{x_i}-N\mu)(y)\\
&=\mathrm{Error}_{\ep_2,\ua_N}+\int_{\R^d} (\nabla v_{\ep_2})(x) : \comm{H_{N,\ua_N}^{\mu,\ux_N}}{H_{N,\ua_N}^{\mu,\ux_N}}_{SE}(x)dx \\
&\ph -\sum_{i=1}^N\int_{(\R^d)^2} \paren*{v_{\ep_2}(x)-v_{\ep_2}(y)}\cdot(\nabla\g)(x-y)d\d_{x_i}^{(\alpha_i)}(x)d\d_{x_i}^{(\alpha_i)}(y),
\end{split}
\end{equation}
where $\mathrm{Error}_{\ep_2,\ua_N}$ satisfies the bound \eqref{eq:recomb_err_bnd}. We choose $\ua_N=\ur_{N,\ep_1}$, where the parameter $\ur_{N,\ep_1}$ is defined in \eqref{eq:r_def}.

We now use the error bound \eqref{eq:recomb_err_bnd} with $\ua_N=\ur_{N,\ep_1}$. Noting that $r_{i,\ep_1}\leq \ep_1$ by definition together with the increasing property of $r\mapsto r|\ln r|$ for $r\in (0,e^{-1})$ and some algebra, we find that
\begin{equation}
\label{eq:conc_err_fin}
\begin{split}
\left|\mathrm{Error}_{\ep_2,\ur_{N,\ep_1}}\right| &\lesssim_{d} \|v\|_{LL(\R^d)}\ln(\ep_2^{-1})\paren*{\Fr_N(\ux_N,\mu)+N\tl{\g}(\ep_4)+C_{p,d}N^2\|\mu\|_{L^p(\R^d)}\ep_4^{\frac{2p-d}{p}}} \\
&\ph + \frac{N^2\ep_1}{\ep_4^{d-1}}\paren*{\frac{\|v\|_{L^\infty(\R^d)}}{\ep_4}+\|v\|_{LL(\R^d)}\ln(\ep_1^{-1})} + N^2\|v\|_{LL(\R^d)}\|\mu\|_{L^p(\R^d)}^{\frac{(d-1)p}{d(p-1)}}\ep_1\ln(\ep_1^{-1}) \\
&\ph + N^2\|v\|_{L^\infty(\R^d)}\paren*{\|\mu\|_{L^p(\R^d)}\ep_1^{1-\frac{d}{p}} + \|\mu\|_{L^\infty(\R^d)}\ep_1\ln(\ep_1^{-1})1_{\geq\infty}(p)}.
\end{split}
\end{equation}

For the third term in the right-hand side of \eqref{eq:conc_RHS}, we use the mean value theorem to obtain the estimate
\begin{align}
&\left|\sum_{i=1}^N\int_{(\R^d)^2} \paren*{v_{\ep_2}(x)-v_{\ep_2}(y)}\cdot(\nabla\g)(x-y)d\d_{x_i}^{(r_{i,\ep_1})}(x)d\d_{x_i}^{(r_{i,\ep_1})}(y)\right| \nn\\
&\lesssim \|\nabla v_{\ep_2}\|_{L^\infty(\R^d)}\sum_{i=1}^N \int_{(\R^d)^2}\g(x_i-x_j)d\d_{x_i}^{(r_{i,\ep_1})}(x)d\d_{x_i}^{(r_{i,\ep_1})}(y) \nn\\
&\lesssim_d \|v\|_{LL(\R^d)}\ln(\ep_2^{-1})\sum_{i=1}^N\tl{\g}(r_{i,\ep_1}),
\end{align}
where the ultimate inequality follows from \cref{lem:g_smear_si}. Using estimate \eqref{eq:g_r_sum_bnd} of \cref{cor:grad_H}, we conclude that
\begin{equation}
\begin{split}
&\left|\sum_{i=1}^N\int_{(\R^d)^2} \paren*{v_{\ep_2}(x)-v_{\ep_2}(y)}\cdot(\nabla\g)(x-y)d\d_{x_i}^{(r_{i,\ep_1})}(x)d\d_{x_i}^{(r_{i,\ep_1})}(y)\right|\\
&\lesssim_d \|v\|_{LL(\R^d)}\ln(\ep_2^{-1})\paren*{\Fr_N(\ux_N,\mu) + N\tl{\g}(\ep_1)+ C_{p,d}N^2\|\mu\|_{L^p(\R^d)}\ep_1^{\frac{2p-d}{p}}}. \label{eq:conc_T2_fin}
\end{split}
\end{equation}

For the second term in the right-hand side of \eqref{eq:conc_RHS}, we use H\"older's inequality, \cref{lem:conv_bnds}, and the point-wise bound
\begin{equation}
|\comm{H_{N,\ur_{N,\ep_1}}^{\mu,\ux_N}}{H_{N,\ur_{N,\ep_1}}^{\mu,\ux_N}}_{SE}(x)| \lesssim |(\nabla H_{N,\ur_{N,\ep_1}}^{\mu,\ux_N})(x)|^2 \qquad \forall x\in\R^d,
\end{equation}
which is immediate from the definition \eqref{eq:se_ten_def} of the stress-energy tensor and Young's inequality for products, in order to obtain
\begin{align}
\left|\int_{\R^d} (\nabla v_{\ep_2})(x) :\comm{H_{N,\ur_{N,\ep_1}}^{\mu,\ux_N}}{H_{N,\ur_{N,\ep_1}}^{\mu,\ux_N}}_{SE}(x)dx\right| &\lesssim \|\nabla v_{\ep_2}\|_{L^\infty(\R^d)}\int_{\R^d} |(\nabla H_{N,\ur_{N,\ep_1}}^{\mu,\ux_N})(x)|^2dx \nn\\
&\lesssim \ln(\ep_2^{-1})\|v\|_{LL(\R^d)}\int_{\R^d} |(\nabla H_{N,\ur_{N,\ep_1}}^{\mu,\ux_N})(x)|^2dx. \label{eq:v_ep_Lip}
\end{align}
Now using estimate \eqref{eq:grad_H_r_bnd} from \cref{cor:grad_H}, we conclude that
\begin{equation}
\label{eq:conc_T1_fin}
\begin{split}
&\left|\int_{\R^d}(\nabla v_{\ep_2})(x) :\comm{H_{N,\ur_{N,\ep_1}}^{\mu,\ux_N}}{H_{N,\ur_{N,\ep_1}}^{\mu,\ux_N}}_{SE}(x)dx\right|\\
&\lesssim_d \ln(\ep_2^{-1})\|v\|_{LL(\R^d)}\paren*{\Fr_N(\ux_N,\mu) + N\tl{\g}(\ep_1) + C_{p,d}N^2\|\mu\|_{L^p(\R^d)}\ep_1^{\frac{2p-d}{p}}}. 
\end{split}
\end{equation}

Combining estimates \eqref{eq:conc_err_fin}, \eqref{eq:conc_T2_fin}, and \eqref{eq:conc_T1_fin}, we obtain that
\begin{equation}
\label{eq:conc_v_ep_fin}
\begin{split}
&\left|\int_{(\R^d)^2\setminus\D_2} \paren*{v_{\ep_2}(x)-v_{\ep_2}(y)}\cdot(\nabla\g)(x-y)d(\sum_{i=1}^N\d_{x_i}-N\mu)(x)d(\sum_{i=1}^N\d_{x_i}-N\mu)(y)\right| \\
&\lesssim_d \|v\|_{LL(\R^d)}\ln(\ep_2^{-1})\paren*{\Fr_N(\ux_N,\mu)+N\tl{\g}(\ep_1)+C_{p,d}N^2\|\mu\|_{L^p(\R^d)}\ep_4^{\frac{2p-d}{p}}} \\
&\ph + \frac{N^2\ep_1}{\ep_4^{d-1}}\paren*{\frac{\|v\|_{L^\infty(\R^d)}}{\ep_4}+\|v\|_{LL(\R^d)}\ln(\ep_1^{-1})} + N^2\|v\|_{LL(\R^d)}\|\mu\|_{L^p(\R^d)}^{\frac{(d-1)p}{d(p-1)}}\ep_1\ln(\ep_1^{-1}) \\
&\ph + N^2\|v\|_{L^\infty(\R^d)}\paren*{\|\mu\|_{L^p(\R^d)}\ep_1^{1-\frac{d}{p}} + \|\mu\|_{L^\infty(\R^d)}\ep_1\ln(\ep_1^{-1})1_{\geq \infty}(p)}.
\end{split}
\end{equation}
Combining estimate \eqref{eq:conc_v_ep_fin} with \cref{lem:kprop_error}, we finally conclude that
\begin{equation}
\begin{split}
&\left|\int_{(\R^d)^2\setminus\D_2} \paren*{v(x)-v(y)}\cdot(\nabla\g)(x-y)d(\sum_{i=1}^N\d_{x_i}-N\mu)(x)d(\sum_{i=1}^N\d_{x_i}-N\mu)(y)\right| \\
&\lesssim_d \|v\|_{LL(\R^d)}\paren*{\ln_+(N^2H_{N,d}) + \ln(\ep_2^{-1})}\Fr_N(\ux_N,\mu) + \|v\|_{LL(\R^d)}N\paren*{\ln(\ep_2^{-1})\tl{\g}(\ep_1)+\ln_+(N^2H_{N,d})\tl{\g}(\ep_3)}\\
&\ph + \|v\|_{LL(\R^d)}C_{p,d}N^2\|\mu\|_{L^p(\R^d)}\paren*{\ln(\ep_2^{-1})\ep_4^{\frac{2p-d}{p}}+\ln_+(N^2H_{N,d})\ep_3^{\frac{2p-d}{p}}} \\
&\ph + \frac{N^2\ep_1}{\ep_4^{d-1}}\paren*{\frac{\|v\|_{L^\infty(\R^d)}}{\ep_4} + \|v\|_{LL(\R^d)}\ln(\ep_1^{-1})} + N^2\|v\|_{LL(\R^d)}\|\mu\|_{L^p(\R^d)}^{\frac{(d-1)p}{d(p-1)}}\ep_1\ln(\ep_1^{-1}) \\
&\ph + N^2\|v\|_{LL(\R^d)}\ep_2\ln(\ep_2^{-1})\paren*{\frac{1}{\ep_3^{d-1}}+\|\mu\|_{L^p(\R^d)}^{\frac{(d-1)p}{d(p-1)}}} \\
&\ph + N^2\|v\|_{L^\infty(\R^d)}\paren*{\|\mu\|_{L^p(\R^d)}\ep_1^{1-\frac{d}{p}}+\|\mu\|_{L^\infty(\R^d)}\ep_1\ln(\ep_1^{-1})1_{\geq\infty}(p)},
\end{split}
\end{equation}
which completes the proof of \cref{prop:kprop}.

\section{Proof of Main Results}
\label{sec:MR}
In this last section, we prove our main results \cref{thm:main} and \cref{cor:main}. We first record a lemma giving the precise statement for the time derivative of the modulated energy $\Fr_N$. We leave filling in the details of the proof of the lemma as an exercise for the interest reader.

\begin{lemma}[Modulated energy derivative]
\label{lem:en_ineq}
Let $d\geq 3$ and $N\in\N$. Let $\ux_N\in C^\infty([0,\infty);\R^d\setminus\D_N)$ be a solution to the system \eqref{eq:Cou_sys}, and let $\omega \in L^\infty([0,T]; \P(\R^d)\cap L^p(\R^d))$, for some $d<p\leq\infty$, be a weak solution to the equation \eqref{eq:Cou_pde}. Then $\Fr_N^{avg}(\ux_N,\mu): [0,\infty)\rightarrow [0,\infty)$ is locally Lipschitz continuous, and for a.e. $t\in [0,T]$, we have the identity
\begin{equation}
\label{eq:en_ineq}
\frac{d}{dt}\Fr_N^{avg}(\ux_N(t),\omega(t)) = \int_{(\R^d)^2\setminus\D_2} (\nabla\g)(x-y)\cdot\paren*{u(t,x)-u(t,y)}d(\omega_N-\omega)(t,x)d(\omega_N-\omega)(t,y),
\end{equation}
where $u$ is the velocity field associated to $\omega$ through the Biot-Savart law and $\omega_N\coloneqq \frac{1}{N}\sum_{i=1}^N \d_{x_i}$ is the empirical measure associated to $\ux_N$.
\end{lemma}
\begin{proof}
See the proof of \cite[Lemma 5.1]{Rosenzweig2020_PVMF}, which treats the $d=2$ case, and whose arguments generalize to higher dimensions.
\end{proof}

\subsection{Proof of \cref{thm:main}}
\label{ssec:MR_thm}
We now prove \cref{thm:main} using \cref{lem:en_ineq}. By the fundamental theorem of calculus for Lebesgue integrals and the triangle inequality, we have that
\begin{equation}
\begin{split}
\left|\Fr_N^{avg}(\ux_N(t),\omega(t))\right| &\leq \left|\Fr_N^{avg}(\ux_N(0),\om(0))\right| \\
&\ph + \int_0^t \left|\int_{(\R^d)^2\setminus\D_2} (\nabla\g)(x-y)\cdot\paren*{u(s,x)-u(s,y)}d(\omega-\omega_N)(s,x)d(\omega-\omega_N)(s,y)\right|ds.
\end{split}
\end{equation}
Applying \cref{prop:kprop} with $p=\infty$ to the integrand point-wise in $s$ and using the bounds
\begin{align}
\|u\|_{L^\infty([0,T];L^\infty(\R^d))} &\lesssim_d \|\om^0\|_{L^\infty(\R^d)}^{\frac{d-1}{d}} \\
\|u\|_{L^\infty([0,T];LL(\R^d))} &\lesssim_d \|\om^0\|_{L^\infty(\R^d)},
\end{align}
which follow by \cref{lem:PE_bnds} and conservation of the $L^\infty$ norm, there exists a constant $C_{1,d}>0$ such that
\begin{equation}
\label{eq:ep_pre}
\begin{split}
|\Fr_N^{avg}(\ux_N(t),\om(t))| &\leq |\Fr_N^{avg}(\ux_N(0),\om(0))| + C_{1,d}\|\om^0\|_{L^\infty(\R^d)}\paren*{\ln_+(N^2H_{N,d})+\ln(\ep_2^{-1})}\int_0^t |\Fr_N^{avg}(\ux_N(s),\om(s))|ds \\
&\ph + \frac{C_{1,d}t\|\om^0\|_{L^\infty(\R^d)}}{N}\paren*{\frac{\ln(\ep_2^{-1})}{\ep_1^{d-2}}+\frac{\ln_+(N^2H_{N,d})}{\ep_3^{d-2}}}\\
&\ph + C_{1,d}t\|\om^0\|_{L^\infty(\R^d)}^{2}\paren*{\ln(\ep_2^{-1})\ep_4^2 + \ln_+(N^2H_{N,d})\ep_3^2} \\
&\ph + C_{1,d}t\paren*{\frac{\|\om^0\|_{L^\infty(\R^d)}^{\frac{d-1}{d}}\ep_1}{\ep_4^d} + \frac{\|\om^0\|_{L^\infty(\R^d)}\ep_1\ln(\ep_1^{-1})}{\ep_4^{d-1}} + \|\om^0\|_{L^\infty(\R^d)}^{\frac{2d-1}{d}}\ep_1\ln(\ep_1^{-1})} \\
&\ph + C_{1,d}t\|\om^0\|_{L^\infty(\R^d)}\paren*{\frac{\ep_2\ln(\ep_2^{-1})}{\ep_3^{d-1}} + \|\om_0\|_{L^\infty(\R^d)}^{\frac{d-1}{d}}\ep_2\ln(\ep_2^{-1})}.
\end{split}
\end{equation}
We choose $\ep_4=\ep_1^{1/(d+2)}$, $\ep_3=\ep_2^{1/(d+1)}$, $\ep_2=\ep_1^{(d+1)/(d+2)}$, and $\ep_1=N^{-(d+2)/(d^2-2)}$.\footnote{Note that in contrast to \cite{Rosenzweig2020_PVMF}, there is no need here to take the $\ep_j$ to be time-dependent.} Substituting these choices into \eqref{eq:ep_pre} and simplifying, we find that there exists a constant $C_{2,d}>0$ such that
\begin{equation}
\label{eq:ep_post}
\begin{split}
|\Fr_N^{avg}(\ux_N(t),\om(t))| &\leq |\Fr_N^{avg}(\ux_N(0),\om(0))| +C_{1,d}\|\om^0\|_{L^\infty(\R^d)}\paren*{\ln_+(N^2H_{N,d})+\ln(N)}\int_0^t |\Fr_N^{avg}(\ux_N(s),\om(s))|ds \\
&\ph + C_{2,d}t\paren*{\|\om^0\|_{L^\infty(\R^d)} + \|\om^0\|_{L^\infty(\R^d)}^2}\frac{\ln(N) + \ln_+(N^2H_{N,d})}{N^{\frac{2}{d^2-2}}}.
\end{split}
\end{equation}
By the Gronwall-Bellman inequality, it follows that
\begin{align}
|\Fr_N^{avg}(\ux_N(t),\om(t))| &\leq \G_{d,N}(t)\exp{C_{1,d}t\|\om^0\|_{L^\infty(\R^d)}\paren*{\ln_+(N^2H_{N,d})+\ln(N)}},
\end{align}
where
\begin{equation}
\G_{d,N}(t)\coloneqq |\Fr_N^{avg}(\ux_N(0),\om(0))| + C_{2,d}t\paren*{\|\om^0\|_{L^\infty(\R^d)} + \|\om^0\|_{L^\infty(\R^d)}^2}\frac{\ln(N) + \ln_+(N^2H_{N,d})}{N^{\frac{2}{d^2-2}}}.
\end{equation}
Thus the proof of \cref{thm:main} is complete.

\subsection{Proof of \cref{cor:main}}
\label{ssec:MR_cor}
We now prove \cref{cor:main}.

\begin{proof}[Proof of {\cref{cor:main}}]
Fix $s<-d/2$. If there exists $\eta>0$ such that $|\Fr_N^{avg}(\ux_N(0),\om(0))|\leq C_\et N^{-\et}$, then
\begin{equation}
\sup_{N\in\N} |H_{N,d}| \lesssim_\eta |H_d|,
\end{equation}
where we recall that $H_d$ is the Hamiltonian from \eqref{eq:Cou_pde_ham}. So by \cref{prop:CE_coer} and \cref{thm:main},
\begin{align}
\|\om-\om_N\|_{L^\infty([0,T];H^s(\R^d))} &\lesssim_{s,d} (1+\|\om^0\|_{L^\infty(\R^d)})^{1/2}N^{-1/3} + N^{\frac{2s+d}{3(1-s)}} +  \sup_{0\leq t \leq T}|\Fr_N^{avg}(\ux_N(t),\om(t))|^{1/2} \nn\\
&\lesssim_{d,\eta} (1+\|\om^0\|_{L^\infty(\R^d)})^{1/2}N^{-1/3} + N^{\frac{2s+d}{3(1-s)}} \nn\\
&\ph + \paren*{N^{-\et}+T(\|\om^0\|_{L^\infty(\R^d)}+\|\om^0\|_{L^\infty(\R^d)}^2)\frac{(\ln(N)+\ln(H_{d}))}{N^{\frac{2}{d^2-2}}}}^{1/2} \nn \\
&\ph \hspace{15mm} \times N^{C_dT\|\om^0\|_{L^\infty(\R^d)}}e^{C_{d,\et}\|\om^0\|_{L^\infty(\R^d)}|H_d|T}. \label{eq:T_pre_RHS}
\end{align}
For
\begin{equation}
0< C_d\|\om^0\|_{L^\infty(\R^d)}T < \min\{\frac{1}{d^2-2},\frac{\et}{2}\},
\end{equation}
the right-hand side of \eqref{eq:T_pre_RHS} is $\lesssim_{d,s,\eta}$ a quantity which tends to zero as $N\rightarrow\infty$. Further restricting $T$ improves the rate of convergence. Weak-* convergence now follows by a standard approximation argument. We omit the details.
\end{proof}

\bibliographystyle{siam}
\bibliography{PointVortex}
\end{document}